\documentclass[a4paper,11pt]{article}
\usepackage{amsmath,amssymb,amsthm}
\usepackage{cases}
\usepackage{a4wide}
\usepackage[pdftex]{graphicx}
\usepackage[hang,small,bf]{caption}
\usepackage[subrefformat=parens]{subcaption}
\usepackage{here}
\usepackage{comment}
\usepackage{ulem}
\usepackage[pdftex]{color}

\usepackage{graphicx} 
\numberwithin{equation}{section}

\newtheorem{prop}{Proposition}[section]
\newtheorem{thm}{Theorem}[section]

\newtheorem{lem}{Lemma}[section]
\newtheorem{rem}{Remark}[section]
\newtheorem{cor}{Corollary}[section]

\newtheorem{dfn}{Definition}[section]

\newcommand{\Figure}[1]{Figure~\ref{#1}}
\newcommand{\Theorem}[1]{Theorem~\ref{#1}}
\newcommand{\Proposition}[1]{Proposition~\ref{#1}}
\newcommand{\Remark}[1]{Remark~\ref{#1}}

\newcommand{\curve}[2]{\Gamma^{#1}_{#2}}

\newcommand{\closure}[1]{\overline{#1}}

\newcommand{\grad}[1]{\nabla #1}

\newcommand{\Lemma}[1]{Lemma~\ref{#1}}

\newcommand{\laplacian}[1]{\Delta{#1}}

\newcommand{\pOmega}[0]{\partial\Omega}

\newcommand{\weaksol}[2]{w^{#1}_{#2}}

\newcommand{\dH}[1]{d\mathcal{H}^{#1}}

\newcommand{\halfplane}[1]{\mathbb{R}^{#1}_+}

\newcommand{\subdiff}[2]{\partial #1 (#2)}

\newcommand{\argmin}{\operatorname{argmin}}

\newcommand{\dL}[1]{\,\mathrm{d}\mathcal{L}^{#1}}

\newcommand{\script}[2]{\mathcal{#1}_{#2}}
\newcommand{\geodis}[1]{d_{\Omega,#1}}

\newcommand{\zerovec}[0]{\bvzero}

\def\bv#1{\mbox{\boldmath{$#1$}}}    

\def\bvzero{{\bf 0}}

\newcommand{\innerproductSingle}[2]{\left<#1,#2\right>}
\newcommand{\chambolleEnergyFour}[1]{E^\beta_h(#1)}
\newcommand{\chambolleEnergyFourOriginal}[1]{E_h(#1)}
\newcommand{\projection}[2]{\pi_{#1}(#2)}
\newcommand{\projectionnoparam}[1]{\pi_{#1}}
\newcommand{\norm}[1]{\left\lVert#1\right\rVert}

\newcommand{\hess}[1]{\nabla^2 #1}

\newcommand{\relaxsup}[2]{\operatornamewithlimits{limsup^*}_{#1\to #2}}
\newcommand{\relaxinf}[2]{{\operatornamewithlimits{liminf}}_{*#1\to #2}}
\newcommand{\approxscheme}[2]{S_{#2}}
\newcommand{\chambolleSet}[1]{T_{#1}}
\newcommand{\chambolleFunction}[1]{S_{#1}}
\newcommand{\del}[2]{\partial_{#1}{#2}}
\newcommand{\normalConv}[0]{\nu}
\newcommand{\supersolution}[1]{\overline{#1}}
\newcommand{\subsolution}[1]{\underline{#1}}
\newcommand{\sublevel}[2]{E^{#1}_{#2}}
\newcommand{\solitonprofile}[2]{\Phi_{#1}(#2)}
\newcommand{\soliton}[1]{s_{#1}}
\newcommand{\anisotropy}[3]{\phi_{#1#2}(#3)}
\newcommand{\betahigh}[0]{\overline{\beta}}
\newcommand{\betalow}[0]{\underline{\beta}}
\newcommand{\chigh}[0]{\overline{C}}
\newcommand{\clow}[0]{\underline{C}}
\newcommand{\gradbdd}[2]{\nabla_{#1}#2}
\newcommand{\containerTangent}[0]{\tau_\Omega}
\newcommand{\surfacenormal}[2]{\frac{\grad{#1(#2)}}{|\grad{#1}(#2)|}}

\title{{A c}onvergence result for a minimizing movement scheme {for mean curvature flow} with prescribed contact angle in a curved domain}
\author{Tokuhiro Eto\thanks{Graduate School of Mathematical Sciences, The University of Tokyo, Komaba 3-8-1, Meguro, Tokyo 153-8914, Japan. E-mail:tokuhiro.eto@gmail.com} \and Yoshikazu Giga \thanks{Graduate School of Mathematical Sciences, The University of Tokyo, Komaba 3-8-1, Meguro, Tokyo 153-8914, Japan. E-mail:labgiga@ms.u-tokyo.ac.jp}}
\date{\today}

\begin{document}

\maketitle

\begin{abstract}
We consider a minimizing movement scheme of Chambolle type for the mean curvature flow equation with prescribed contact angle condition in a smooth bounded domain in $\mathbb{R}^d$ ($d\geq2$).
 We prove that an approximate solution constructed by the proposed scheme converges to the level-set mean curvature flow with prescribed contact angle provided that the domain is convex and that the contact angle is away from zero under some control of derivatives of given prescribed angle.
 We actually prove that an auxiliary function corresponding to the scheme uniformly converges to a unique viscosity solution to the level-set equation with an oblique {derivative} boundary condition corresponding to the prescribed boundary condition.
\end{abstract}

\section{Introduction}
Let $\Omega$ be a smooth bounded domain in $\mathbb{R}^d$ with $d\geq 2$.
We consider the mean curvature flow equation for an evolving family $\{\Gamma_t\}_{t\geq 0}$ of hypersurfaces in $\closure{\Omega}$
touching the boundary $\pOmega$ of $\Omega$ with prescribed angle. Namely,
we consider its initial value problem of the form:
\begin{equation}\label{eq:4_Target}
    \begin{cases}
        V = -\operatorname{div}_{\Gamma_t}\normalConv\ \ \mbox{on}\ \ \Gamma_t,\ \ t{>} 0,\\
        \angle(\normalConv,\normalConv_\Omega) = \theta\ \ \mbox{on}\ \ \pOmega,\\
        \Gamma_0 = \Gamma,
    \end{cases}
\end{equation}
where $\Gamma$ is a given hypersurface in $\closure{\Omega}$.
{
    Here, $\nu$ denotes a unit vector field of $\Gamma_t$, and $V$ denotes the normal velocity
    of each point on $\Gamma_t$ in the direction of $\nu$; $\operatorname{div}_{\Gamma_t}$ denotes the
    surface divergence of $\nu$ on $\Gamma_t$ so that $-\operatorname{div}_{\Gamma_t}{\nu}$ equals
    the ($(d-1)$-times) mean curvature of $\Gamma_t$ in the direction of $\nu$.
    Thus, the first equation of \eqref{eq:4_Target} is nothing but the mean curvature flow equation.
    The second equation of \eqref{eq:4_Target} is the boundary condition.
    The symbol $\angle(\nu,\nu_\Omega)$ denotes the angle between $\nu$ and $\nu_\Omega$,
    where $\nu_\Omega$ denotes the outward unit normal vector field of $\pOmega$.
    The function $\theta:\pOmega\to(0,\pi)$ is given and prescribes the contact angle $\angle(\nu,\nu_\Omega)$.
}
The authors \cite{EtoGiga2023} extended Chambolle's scheme \cite{Chambolle2004} to construct an approximate solution to
$\eqref{eq:4_Target}$. In the sequel, the proposed scheme will be referred as a capillary Chambolle type scheme.
We briefly review the proposed scheme together with some literature. 
Given an initial data $E_0\subset\mathbb{R}^d$ and a time step $h>0$,
Almgren, Taylor and Wang \cite{AlmregTaylorWang1993} introduced the following energy functional:
\begin{equation}\label{eq:4_ATWEnergy}
    \script{A}{}(F) := \int_{\mathbb{R}^d}|\grad{\chi}_F| + \frac{1}{h}\int_{F\Delta E_0}\operatorname{dist}(\cdot,\partial E_0)\dL{d},
\end{equation}
where
    $\chi_F$ denotes the characteristic function of $F$, i.e.,
    $\chi_F(x) = 1$ if $x\in F$ and $\chi_F(x) = 0$ if $x\notin F$,
    and $\operatorname{dist}(x,A)$ denotes the distance between a point $x$ and a set $A$;
    $h$ is a positive parameter, and $F\Delta E_0$ denotes the symmetric difference of $F$ and ${E_0}$, namely
    $F\Delta E_0:=(F\backslash E_0)\cup(E_0\backslash F)${;}
    the first term {of \eqref{eq:4_ATWEnergy}} denotes the total variation of $\chi_F$ while in the second term, $\mathcal{L}^d$ denotes
    the $d$-dimensional Lebesgue measure.
They minimized $\script{A}{}(F)$ among all Caccioppoli sets $F$, and 
its minimizer $\chambolleSet{h}(E_0)$ was regarded as a candidate
for the next set to $E_0$. Repeating this process, 
an approximate solution of the mean curvature flow was defined by 
$\chambolleSet{h}^n(E_0) := \chambolleSet{h}(\chambolleSet{h}^{n-1}(E_0))$ for $n\in\mathbb{N}$ with $\chambolleSet{h}^0(E_0) := E_0$.
{Alt}hough they showed its convergence to the smooth mean curvature flow in $L^1$-setting, {it is not clear whether a minimizer of $\script{A}{}(F)$ is unique or not.}
{Later}, Chambolle \cite{Chambolle2004} proposed another energy functional defined by:
\begin{equation*}
    \chambolleEnergyFourOriginal{u} := \int_{\Omega}|\grad{u}| + \frac{1}{2h}\int_{\Omega}(u - d_{E_0})^2\,\dL{d},
\end{equation*}
where $d_{E_0}$ denotes the signed distance function to $E_0$.
The energy $\chambolleEnergyFourOriginal{u}$ was minimized over all $u\in L^2(\Omega)\cap BV(\Omega)$.
Since it is lower semi-continuous and strictly convex, the minimizer $\weaksol{h}{E_0}$ is unique.
Then, the set $\chambolleSet{h}(E_0)$ was defined by the zero sub-level set of $\weaksol{h}{E_0}$.
Chambolle {\cite{Chambolle2004}} showed that {his $T_h(E)$ is a minimizer of $\script{A}{}(F)$, and} the approximate solution tends to the mean curvature flow in $L^1$-setting
if the corresponding {level-set} equation with the initial condition $u_0:=\chi_{\Omega\backslash E_0} - \chi_{E_0}$ has
a unique viscosity solution where $\chi_{E_0}$ is the characteristic function of $E_0$.
This scheme worked only if {$E$ does not touch the boundary $\pOmega$} hence a contact angle condition cannot be treated.

To cope {with} the contact angle problem, the authors \cite{EtoGiga2023} proposed a capillary Chambolle type scheme 
to construct an approximate solution to the mean curvature flow with prescribed contact angle condition. 
For $\beta\in L^\infty(\pOmega)$ with $\|\beta\|_\infty\leq 1$,
they alternately solved the following variational problem:
\begin{equation}\label{eq:4_MinimizingEnergy}
    \min_{u}{E^\beta_h(u)}\qquad\mbox{{with}}\qquad {E^\beta_h(u) :=}\ C_\beta(u) + {\frac{1}{2h}}\int_\Omega(u - d_{\Omega,E_0})^2\,{\dL{d}}.
\end{equation}
Here, the minimum $\eqref{eq:4_MinimizingEnergy}$ should be taken over $L^2(\Omega)\cap\, BV(\Omega)${;}
$d_{\Omega,E_0}$ denotes the signed geodesic distance function to $E_0$ in $\Omega$ (see e.g. \cite[Definition 4]{EtoGiga2023}), and 
\begin{equation*}
    C_\beta(u) := \int_\Omega|\nabla u| + \int_{\pOmega}\beta\gamma u\,\dH{d-1},
\end{equation*}
where $\gamma:BV(\Omega)\to L^1(\pOmega)$ is the trace operator {and $\mathcal{H}^{d-1}$ denotes the $(d-1)$-dimensional Hausdorff measure}.
{For \eqref{eq:4_Target}, we take $\beta = \cos\theta$ and $\theta\in(0,\pi)$ which imply $|\beta(x)| < 1$.}
Since $C_\beta$ is lower semi-continuous (see \cite[Proposition 1.2]{Modica1987} and \cite[Proposition 4]{EtoGiga2023}) 
and the quantity to be minimized in $\eqref{eq:4_MinimizingEnergy}$ is strictly convex in $L^2(\Omega)$, the problem $\eqref{eq:4_MinimizingEnergy}$ admits a unique minimizer $\weaksol{h}{E_0}\in L^2(\Omega)\cap BV(\Omega)$. 
Then, the next set $T_h(E_0)$ to $E_0$ is defined as the zero sub-level set of $\weaksol{h}{E_0}$, namely $T_h(E_0) := \{\weaksol{h}{E_0} \leq 0\}$.
They established well-posedness of the proposed scheme and some consistency with a capillary Almgren--Taylor--Wang type scheme.
However, {the convergence of this scheme as $h$ tends to zero was not discussed  in \cite{EtoGiga2023}.}
The aim of this paper is to show {the} convergence of {this} scheme in {a suitable topology}.

Let us state our main results. 
{
    Since the level-set formulation provides a unique global-in-time solution (up to fattening) as shown in \cite{ChenGigaGoto1991,ES} for $\Omega = \mathbb{R}^d$ (see also \cite{G}),
    we consider the level-set formulation of \eqref{eq:4_Target}. Namely,}
we consider the initial-boundary value problem {for its level-set equation of the form}:
\begin{equation}\label{eq:4_LevelSetEquationBdd}
    \begin{cases}
        u_t = |\nabla u|\operatorname{div}\grad{\anisotropy{}{}{\grad{u}}}\ \ \mbox{in}\ \ \Omega\times(0,T),\\
        \innerproductSingle{\grad{u}}{\normalConv_\Omega} + \beta|\nabla u| = 0 \ \ \mbox{on}\ \ \pOmega\times(0,T),\\
        u(0,\cdot) = u_0\ \ \mbox{in}\ \ {\closure{\Omega}}, \\
    \end{cases}
\end{equation}
where {$T > 0$ is a time horizon;} $u_0:\closure{\Omega}\to\mathbb{R}$ is given as an initial condition and $\phi(p) := |p|$ for $p\in\mathbb{R}^d$.
It is natural to consider such an oblique derivative boundary problem because 
the second condition of $\eqref{eq:4_LevelSetEquationBdd}$ implies that the hypersurface $\{u = 0\}$ 
intersects {the boundary $\pOmega$} with the angle $\arccos{\beta}$. This condition readily corresponds to the second one of $\eqref{eq:4_Target}$.
{Let} $F:(\mathbb{R}^d\backslash{\{\zerovec\}})\times\mathbb{S}^d\to\mathbb{R}$ {and $B:\pOmega\times\mathbb{R}^d\to\mathbb{R}$ be} defined by
\begin{align}
    &F(p,X) := -\operatorname{tr}\left(\left(I_d - \frac{p\otimes p}{|p|^2}\right)X\right)\ \ \mbox{for}\ \  (p,X)\in (\mathbb{R}^d\backslash{\{\zerovec\}})\times\mathbb{S}^d,\label{eq:intro_F_formula}\\
    &{B(x,p) := \innerproductSingle{p}{\nu_\Omega(x)} + \beta(x)|p|},\label{eq:intro_B_formula}
\end{align}
where {$\mathbb{S}^d$ denotes the set of all symmetric matrices in $\mathbb{R}^{d\times d}$;}
$I_d\in\mathbb{R}^{d\times d}$ denotes the identity matrix. 
{
    Then, the problem \eqref{eq:4_LevelSetEquationBdd} can be expressed as
    \begin{equation}\label{eq:4_LevelSetEquationBdd_General}
        \begin{cases}
            u_t + F(\grad{u},\hess{u}) = 0\qquad\mbox{in}\quad\Omega\times(0,T),\\
            B(\cdot, \grad{u}) = 0\qquad\mbox{on}\quad\pOmega\times(0,T),\\
            u(\cdot,0) = u_0\qquad\mbox{in}\quad\closure{\Omega}.
        \end{cases}
    \end{equation}
}

In this study, we adopt the notion of viscosity solution{s} and regard an evolving set 
by mean curvature as the level set of an auxiliary function as discussed in {\cite{G}}.
{Its well-posedness is by now well known by \cite{Ba,IS}.}
{As in \cite{EtoGigaIshii2012,EtoGigaIshii2012_2}, f}or each $u\in UC(\closure{\Omega})$ and a time step $h>0$, we define a function operator {$\chambolleFunction{h}$} by
\begin{equation}\label{eq:T_h_S_h_formula}
    \chambolleFunction{h}u(x) := \sup\{\lambda\in\mathbb{R}\mid x\in \chambolleSet{h}(\{u\geq \lambda\})\}{,}
\end{equation}
{where $UC(\closure{\Omega})$ denotes the space of all uniformly continuous functions in $\closure{\Omega}$.}
In terms of $\chambolleFunction{h}$, an approximate solution $u^h:[0,T]\times\closure{\Omega}\to\mathbb{R}$ to $\eqref{eq:4_LevelSetEquationBdd}$ is defined by 
\begin{equation*}
    u^h(t,x) := \chambolleFunction{h}^{\lfloor\frac{t}{h}\rfloor}u(x),
\end{equation*}
where $\lfloor k\rfloor$ denotes the largest integer which does not exceed $k\in(0,\infty)$.
Then, our main theorem reads as follows:
\begin{thm}\label{thm:final_statement}
    Assume that $\Omega$ is a bounded convex set in $\mathbb{R}^d$
    whose boundary is sufficiently regular {so that} the comparison principle {holds for} \eqref{eq:4_LevelSetEquationBdd}.
    Suppose that $\beta\in C^1(\pOmega)$ {and $\|\beta\|_\infty < 1$}.
    Assume that $|\gradbdd{\pOmega}{\beta}(x)|\leq k(x)$ for all $x\in\pOmega$. Here, $k(x)$ denotes the minimal {nonnegative} {principal (inward)} curvature of $\pOmega$ at the point $x$.
    Then, $u^h$ uniformly converges to the unique viscosity solution to $\eqref{eq:4_LevelSetEquationBdd}$ as $h\to 0$.
\end{thm}
{
    By a comparison principle for \eqref{eq:4_LevelSetEquationBdd} known by \cite{Ba} ({see} \Theorem{thm:4_ComparisonPrinciple}),
    we immediately obtain the following corollary:
    \begin{cor}
        Let $\Omega\subset\mathbb{R}^d$ be a $C^{2,1}$ bounded domain.
        Suppose that $\Omega$ and $\beta$ satisfy the hypotheses in \Theorem{thm:final_statement}.
        Then, $u^h$ uniformly converges to the unique viscosity solution to \eqref{eq:4_LevelSetEquationBdd} as $h\to 0$.
    \end{cor}
}

A key step of the proof for {\Theorem{thm:final_statement}} is to confirm 
that the function operator $\chambolleFunction{h}$ fulfills the following properties:\newline\newline
    [\textbf{Monotonicity}]\newline
    \begin{equation}\label{eq:4_Monotonicity}
        \approxscheme{}{h}u\leq \approxscheme{}{h}v\ \ \mbox{if}\ \ u\leq v{.}
    \end{equation}
    [\textbf{Translation invariance}]
    \begin{align}\label{eq:4_TranslationInvariance}
        \begin{split}
            &\approxscheme{}{h}(u+c) = \approxscheme{}{h}u + c\ \ \mbox{for {all}}\ \ c\in\mathbb{R},\\
            &\approxscheme{}{h}(0) = 0{.}
        \end{split}
    \end{align}
    [\textbf{Consistency}]\newline
    For every $\varphi\in C^2(\closure{\Omega})$, and ${z}\in\Omega$ either
    $\grad{\varphi}({z})\neq\zerovec$ or 
    $\grad{\varphi}({z}) =\zerovec$ and $\hess{\varphi}({z}) = O$, and 
    ${z}\in\pOmega$ with $\innerproductSingle{\grad{\varphi}({z})}{\normalConv_\Omega({z})} + \beta{(z)}|\grad{\varphi}({z})| >0$, it holds that
    \begin{equation}\label{eq:4_consistency_super}
        \relaxsup{h}{0}\frac{\approxscheme{}{h}\varphi({z}) - \varphi({z})}{h} \leq -F_*(\grad{\varphi}({z}),\hess{\varphi}({z})).
    \end{equation}
    Moreover, for every $\varphi\in C^2(\closure{\Omega})$, and ${z}\in\Omega$ either
    $\grad{\varphi}({z})\neq\zerovec$ or 
    $\grad{\varphi}({z}) =\zerovec$ and $\hess{\varphi}({z}) = O$, and 
    ${z}\in\pOmega$ with $\innerproductSingle{\grad{\varphi}({z})}{\normalConv_\Omega({z})} + \beta{(z)}|\grad{\varphi}({z})| < 0$, it holds that
    \begin{equation}\label{eq:4_consistency_sub}
        \relaxinf{h}{0}\frac{\approxscheme{}{h}\varphi({z}) - \varphi({z})}{h} \geq -F^*(\grad{\varphi}({z}),\hess{\varphi}({z})).
    \end{equation}
Here, {
    for a function $F_h$ on $\closure{\Omega}$ which is parametrized by $h>0$,
    we have used the notation that for $x\in\closure{\Omega}$,
    \begin{align*}
        &\relaxsup{h}{0}F_h(x) := \lim_{h\to 0} \sup\{F_h(y)\mid |x-y|<{\delta,\, 0<\delta< h}\},\\
        &\relaxinf{h}{0}F_h(x) := \lim_{h\to 0} \inf\{F_h(y)\mid |x-y|<{\delta,\, 0<\delta< h}\}.
    \end{align*}
}
Moreover, we define the upper(resp, lower) semi-continuous envelope $F^*$ (resp, $F_*$) of $F$ by
\begin{align*}
    &F^*(p,X) := \lim_{\varepsilon\to 0}\sup\{F(q,Y)\mid |p - q|<\varepsilon,\ \|X - Y\|_{{2}}<\varepsilon\},\\
    &F_*(p,X) := \lim_{\varepsilon\to 0}\inf\{F(q,Y)\mid |p - q|<\varepsilon,\ \|X - Y\|_{{2}}<\varepsilon\}{,}
\end{align*}
{
where for a matrix $X = (x_{ij})_{1\leq i,j\leq d}\in\mathbb{R}^{d\times d}$, 
$\|X\|_2$ denotes the Hilbert--Schmidt norm of $X$, i.e.,
$\|X\|_2 := \sqrt{\sum_{i,j=1}^d x_{ij}^2}$.
}

{
    If the limit equation \eqref{eq:4_LevelSetEquationBdd} has a comparison principle and $S_h$ satisfies the
    conditions \eqref{eq:4_Monotonicity}--\eqref{eq:4_consistency_sub}, a general theory for the monotone scheme \cite[Theorem 2.1]{BarlesSouganidis1991}
    yields the desired result. In our case, we know \eqref{eq:4_LevelSetEquationBdd} has a comparison principle (\Theorem{thm:4_ComparisonPrinciple})
    so the main task is to prove \eqref{eq:4_Monotonicity}--\eqref{eq:4_consistency_sub} for $S_h$ (\Theorem{thm:4_S_h_satisfies_MTC}).
    It is not difficult to prove \eqref{eq:4_Monotonicity} if we take the geodesic distance in \eqref{eq:4_MinimizingEnergy}.
    The property \eqref{eq:4_TranslationInvariance} is easy to confirm from the definition of $S_h$.
    Main efforts for the convergence result of our scheme are devoted to prove \eqref{eq:4_consistency_super} and \eqref{eq:4_consistency_sub}.
}

{
    To this end, we first establish a relation between $S_h$ and $T_h$ (\Lemma{lem:4_Sh_Th}) to interpret 
    the formulae \eqref{eq:4_consistency_super} and \eqref{eq:4_consistency_sub} in that for $T_h$.
    {More explicitly}, we show
    \begin{equation}\label{eq:intro_Sh_Th_ls}
        \{S_hu\geq\lambda\} = T_h(\{u\geq\lambda\})\qquad\mbox{for {all}}\quad\lambda\in\mathbb{R}.
    \end{equation}
    The relation \eqref{eq:intro_Sh_Th_ls} is expected to hold because of the definition of $S_h$.
    One of important criteria to ensure \eqref{eq:intro_Sh_Th_ls} is the continuity of $T_h$ (\Lemma{lem:ChambollesetContinuity}) which is defined by
    \begin{equation}\label{eq:intro_Th_continuity}
        \bigcap_{n=1}^\infty T_h(E_n) = T_h(E)\qquad\mbox{as}\quad n\rightarrow\infty,
    \end{equation}
    where $\{E_n\}_{n\in\mathbb{N}}$ is an arbitrary non-increasing sequence of set{s} in $\closure{\Omega}$ {with $E = \bigcap_{n=1}^\infty E_n$}.
    Thanks to the monotonicity of $T_h$, we {easily} see that the left-hand side of \eqref{eq:intro_Th_continuity} includes the right-hand side.
    To show the converse inclusion, we seek a subsequence $\{\weaksol{h}{E_n}\}_n$ which uniformly converges to a function $w$ in $\closure{\Omega}$,
    and observe that $w = \weaksol{h}{E}$ in $\closure{\Omega}$. Since $\weaksol{h}{E_n}$ is uniformly bounded with respect to $n$ by a maximum principle,
    {the existence of such a subsequence can be proved by the Ascoli--Arzel\`a theorem provided that $\weaksol{h}{E_n}$ is equi-continuous with respect to $n\in\mathbb{N}$.}
    We eventually know that the limit function $w$ must equal $\weaksol{h}{E}$ by the Lipschitz continuity of the map
    $L^2(\Omega)\ni g\mapsto\weaksol{h}{g}\in L^2(\Omega)$ (\Proposition{prop:Lipschitz_Continuity})
    and the uniform convergence of $\geodis{E_n}$ to $\geodis{E}$.
}

To derive the equi-continuity of $\weaksol{h}{E_n}$, we show that the gradient $\grad{\weaksol{h}{E_n}}$
is bounded by a constant which is independent of $n\in\mathbb{N}$. For this, we adopt Bernstein's method.
Namely, we are led to show that {$\frac{1}{2}|\grad{\weaksol{h}{E_n}}|^2$} is a subsolution to an elliptic problem with an oblique derivative boundary condition
and to apply a comparison principle (\Lemma{lem:4_comparison_principle}) which is available for the problem.
{This procedure involving Bernstein's method} requires the convexity of $\Omega$ and the assumption that the contact angle function $\beta$ is continuously differentiable on $\pOmega$
and its gradient is bounded by the minimal {nonnegative} principal (inward) curvature of $\pOmega$ at each point.

{
    Using the relation between $S_h$ and $T_h$ that we have obtained so far, we represent the formulae \eqref{eq:4_consistency_super} and \eqref{eq:4_consistency_sub}
    in terms of the super level sets $\sublevel{\varphi}{\mu} := \{\varphi\geq\mu\}$ with $\mu=\varphi(z)$, and intend to prove that
    the following locally uniform limit in $\closure{\Omega}$ {(Proposition \ref{prop:4_existence_super_sub_solution})}:
    \begin{equation}\label{eq:intro_Th_consistency}
        \left|\frac{\weaksol{h}{\sublevel{\varphi}{\mu}} - \geodis{\sublevel{\varphi}{\mu}}}{h} + \kappa_{\sublevel{\varphi}{\mu}}\right|\rightarrow 0\qquad\mbox{as}\quad h\rightarrow 0.
    \end{equation}
    For the case when $z\in\Omega$, \eqref{eq:intro_Th_consistency} indicates that $\weaksol{h}{\sublevel{\varphi}{\mu}}$ approximately solves the inclusion:
    \begin{equation*}
        \frac{w - \geodis{\sublevel{\varphi}{\mu}}}{h} + \subdiff{C_0}{w}\ni 0\qquad\mbox{in}\quad\Omega,
    \end{equation*}
    which is nothing but a discretization of the first equation of \eqref{eq:4_LevelSetEquationBdd}.
    Meanwhile, if $z\in\pOmega$, we need the assumption that $\|\beta\|_\infty < 1$ to construct a viscosity super(resp, sub)solution to \eqref{eq:4_Neumann_problem_master}
    which approximates the following inclusion:
    \begin{equation}\label{eq:intro_inclusion}
        \frac{w -\geodis{\sublevel{\varphi}{\mu}}}{h} + \subdiff{C_\beta}{w} \ni 0\qquad\mbox{in}\quad\closure{\Omega}.
    \end{equation}
    In fact, we deduce from the characterization of $\subdiff{C_\beta}{u}$ (see \cite[Theorem 2]{EtoGiga2023})
    that \eqref{eq:intro_inclusion} is equivalent to that $w$ solves \eqref{eq:4_Neumann_problem_master}.
    In \cite{EtoGigaIshii2012,EtoGigaIshii2012_2}, a viscosity super(resp, sub)solutions were
    constructed by a ball $B(0,R)$ which is expected to approximate $\sublevel{\varphi}{\mu}$ nearby $z$.
    This is because Chambolle's scheme yield{s} another ball $B(0,\widetilde{R})$, and this $\widetilde{R}$ can be
    explicitly computed. Due to the oblique derivative boundary condition, we require another type of subsets in $\closure{\Omega}$
    to approximate $\sublevel{\varphi}{\mu}$. {Here}, we notice that the characterization of $\subdiff{C_\beta}{u}$
    again gives rigorous solutions to \eqref{eq:4_Neumann_problem_master} 
    in a short time when $\beta$ is constant (\Lemma{lem:4_existence_ts}).
    {These} rigorous solutions are so-called translating solitons in the literature.
    In this research, we compare $\sublevel{\varphi}{\mu}$ with these solitons instead of balls.
    By geometry, this soliton is available not only for $z\in\Omega$ but also for $z\in\pOmega$ whenever $\|\beta\|_\infty<1$.
}
This is the basic idea to prove the main theorem.

Let us review existing works related to an energy minimizing scheme for the mean curvature flow.
{
    For the case when interfaces do not touch the boundary,
    Almgren, Taylor and Wang \cite{AlmregTaylorWang1993} {derived several properties of a limit of their approximate solution and called it}
    a \textit{flat $\Phi$ curvature flow}.  {They proved its convergence to a smooth curvature flow up to the time when the latter exists.}
    Later, Luckhaus and Sturzenhecker \cite{LuckhausSturzenhecker} showed its convergence to a distributional solution $\chi$
    to the mean curvature flow under \textit{no mass loss condition}:
    \begin{equation*}
        \int_0^T\int_{\mathbb{R}^d}|\grad{\chi_h}|\to\int_0^T\int_{\mathbb{R}^d}|\grad{\chi}|\qquad\mbox{as}\quad h\to 0,
    \end{equation*}
    where $\chi_h$ denotes the characteristic function of a minimizer of $\script{A}{}(F)$.
    Philippis and Laux \cite{PhilippisLaux2018} proved that this assumption is not necessary for the convergence result
    whenever the initial data $E_0$ is \textit{outward minimizing}, i.e., it holds that
    \begin{equation}\label{eq:mean_convex}
        \int_{\mathbb{R}^d}|\grad{\chi_{E_0}}|\leq \int_{\mathbb{R}^d}|\grad{\chi_F}|\qquad\mbox{if}\quad E_0\subset F.
    \end{equation}
    For instance, if $E_0$ is mean convex, then the condition \eqref{eq:mean_convex} is satisfied.
    For a bounded initial data $E_0$ and $\Omega\subset\mathbb{R}^d$ strictly including $E_0$,
    Chambolle \cite{Chambolle2004} showed that {his approximate solution constructed by the zero level set of the unique minimizer of $E_h(u)$ converges (in $L^1$ sense) to a \textit{level-set flow} (up to fattening).}
    It is known that $T_h(E_0)$ remains convex if $E_0$ is convex (see Caselles and Chambolle \cite{CasellesChambolle2006}).
    For a simple proof of the convergence {(also in Hausdorff distance sense)},
    we refer the reader to Chambolle and Novaga \cite[Proposition 2.1, Proposition 4.1]{ChambolleNovaga2007}.
    For an unbounded initial data $E_0$, the authors and Ishii \cite{EtoGigaIshii2012,EtoGigaIshii2012_2} showed
    the convergence {(up to fattening in the sense of Hausdorff distance)} of {their approximate solution constructed by $E_h(u)$}.
    Therein, they translated the set operator $T_h$
    into a function operator $S_h$ as in \eqref{eq:T_h_S_h_formula} and showed that $T_h$
    is a morphological operator (see e.g., \cite[Definition 4.4]{Cao2003}). They utilized a $\sup$-$\inf$ representation for $S_h$ 
    to obtain its generator as in \eqref{eq:4_consistency_super} and \eqref{eq:4_consistency_sub}.
    In particular, if $E_0$ is the complement of a bounded set, its treatment was explained in \cite[\S 6.2]{ChambolleMoriniPonsiglione2015}.
    Chambolle, Gennaro and Morini \cite{ChambolleDeGennaroMorini2023}
    considered
    {
        such a scheme called a \textit{minimizing movement scheme} for the mean curvature flow with a time-dependent spatially inhomogeneous driving force $f(x,t)$,
        an inhomogeneous anisotropic interfacial energy density $\phi$ and a mobility $\psi(x,\nu(x))$, namely they studied the equation:
        \begin{equation}\label{eq:inhomogeneous_problem}
            \begin{cases}
                V = \psi(x,\nu(x))\{-\operatorname{div}_{\Gamma_t}\nabla_p\phi(x,\nu(x)) + f(x,t)\}\qquad\mbox{for}\quad x\in\Gamma_t,\ t > 0,\\
                \Gamma_0 = \Gamma,
            \end{cases}
        \end{equation}
        where $\Gamma$ is an initial data, and $p\mapsto\phi(x,p)$ is an anisotropy, i.e., a convex, positively $1$-homogeneous function
        with respect to the variable $p\in\mathbb{R}^d$.
        They showed that a limit of approximate solutions constructed by their minimizing movement scheme is a \textit{flat flow}, and
        it is a distributional solution (BV-solution) to the equation \eqref{eq:inhomogeneous_problem} under no mass loss condition (see \cite[Theorem 1.2]{ChambolleDeGennaroMorini2023}).
        Note that their anisotropy for this result includes crystalline anisotropy, i.e., $\phi(x,p)$ need not be $C^1$ in $p$ on $\{|p| = 1\}$,
        typically piecewise linear. They also proved that their approximate solution converges to the level-set flow, 
        and it is also a way to construct a solution of the corresponding level-set flow equation (see \cite[Theorem 1.4]{ChambolleDeGennaroMorini2023}).
        However, crystalline anisotropy was excluded in this result. For spatially homogeneous cases, i.e., $\phi$, $\psi$ and $f$ are independent of $x$,
        a level-set crystalline mean curvature flow equation is well-studied, and its well-posedness was established by \cite{ChambolleMoriniPonsiglione2017,ChambolleMoriniNovagaPonsiglione2019,GigaPozar2016,GigaPozar2018}
        (see also a review paper \cite{GigaPozar2022}). However, the case of spatially inhomogeneous crystalline anisotropy is not yet studied 
        through a perturbed argument by \cite{ChambolleMoriniNovagaPonsiglione2019} which may lead the existence of a solution.
        For homogeneous case, we note that Ishii \cite{IshiiK2014} proved that the minimizing movement scheme converges to the level-set flow
        for crystalline mean curvature flow provided that the corresponding level-set flow equation is well-posed.
        The solution in \cite{GigaPozar2016,GigaPozar2018} was constructed by a solution of an approximate equation, 
        while the solution in \cite{ChambolleMoriniPonsiglione2017,ChambolleMoriniNovagaPonsiglione2019} was constructed by a minimizing movement scheme.
        In \cite{GigaPozar2020}, spatially inhomogeneous driving force term $f(x,t)$ was allowed. The method of \cite{GigaPozar2016,GigaPozar2018,GigaPozar2020} so far
        needs to assume that $\phi$ is piecewise linear but allows nonlinear dependency on the curvature term in $V$.
    }
}

For {the case when interfaces touch the boundary $\pOmega$},
Bellettini and Kholmatov \cite{BellettiniKholmatov2018} considered
a minimizing problem for a variant energy of $\eqref{eq:4_ATWEnergy}$ defined by
\begin{equation}\label{eq:4_CapillaryATWEnergy}
    \script{A}{\beta}(F) := \int_{\halfplane{{d}}}|\grad{\chi_F}| + \int_{\partial\halfplane{{d}}}\beta\gamma\chi_F\dH{d-1} + \frac{1}{h}\int_{\halfplane{{d}}\cap(F\Delta E)}\operatorname{dist}(\cdot,\partial E)\dL{d}{,}
\end{equation}
{where $\halfplane{d} := \mathbb{R}^{d-1}\times(0,\infty)$, and} the energy \eqref{eq:4_CapillaryATWEnergy} was minimized among all Caccioppoli sets in $\halfplane{{d}}$.
They adopted a set theoretic approach and proved that 
a minimizing sequence for \eqref{eq:4_CapillaryATWEnergy} converges to
a generalized minimizing movement (GMM) {(see \cite[Theorem 7.1]{BellettiniKholmatov2018})}.
The GMM {was shown to be} a distributional solution to the mean curvature flow {equation} with a contact angle condition
{provided that the $(d-1)$-dimensional Hausdorff measure of discrete solutions converges to that of the GMM}
 {(see \cite[Theorem 8.6]{BellettiniKholmatov2018})}.
They also showed regularity of the GMM up to the boundary provided that $\beta$ is Lipschitz continuous on $\pOmega$ (see \cite[Theorem 5.3]{BellettiniKholmatov2018}).
In fact, the study \cite{EtoGiga2023} was inspired by their work {to introduce} the definition of the capillary Chambolle type energy \eqref{eq:4_MinimizingEnergy}.
{For a study of the GMM for a partition of $\mathbb{R}^d$ with mobility and driving force, we refer the reader to Bellettini, Chambolle and Kholmatov \cite{BellettiniChambolleKholmatov2021}.}

{
    In \cite[\S 7]{EtoGiga2023}, the authors implemented the proposed scheme using the split Bregman method.
    Therein, they calculated the first variation of $E^\beta_h(u)$ with respect to $\grad{u}$ and $u$ separately
    and obtained a Neumann boundary problem in a strip $\Omega\subset\mathbb{R}^2$. This problem was solved by the finite difference method.
    {
        They gave two examples of open curves with one end on $\{0\}\times\mathbb{R}$ and the other end on $\{2\}\times\mathbb{R}$ where $\Omega = (0,2)\times\mathbb{R}$.
        In the first example, we set $\beta\equiv\cos{\frac{\pi}{4}}$. In the second example, we set
        $\beta = \cos{\frac{3\pi}{4}}$ at $x = 0$ and $\beta = \cos{\frac{\pi}{4}}$ at $x = 2$.
        However, the discrete definition of $\beta$ in the second example was missing. Here, we give it for the reader's convenience:
    }
    \begin{equation*}
        \beta_{i,j} := \begin{cases}
            -\frac{1}{\sqrt{2}} = \cos{\frac{3\pi}{4}}\qquad\mbox{if}\quad j = 1,\\
            \frac{1}{\sqrt{2}} = \cos{\frac{\pi}{4}}\qquad\mbox{if}\quad j = N_x,\\
            0\qquad\mbox{otherwise},
        \end{cases}
    \end{equation*}
    where the strip $\Omega$ is discretized by mesh points $(x_j, y_i)$ with $1\leq i\leq N_y$ and $1\leq j\leq N_x$,
    say {the discretized} $\closure{\Omega}$ {equals} $\bigcup_{i=1}^{N_y}\bigcup_{j=1}^{N_x}\{(x_j,y_i)\}$;
    the symbol $\beta_{i,j}$ denotes the value of $\beta$ at $(x_j, y_i)$.
}

This paper is organized as follows. In Section \ref{sec:4_Preliminaries}, 
we collect basic definitions, notations and facts of convex analysis and viscosity solutions.
In Section \ref{sec:4_CapillaryChambolleTypeScheme}, we recall a capillary Chambolle type scheme which was proposed in \cite{EtoGiga2023}.
Therein, we continue to explore its properties which are crucial to derive main results in this study.
Section \ref{sec:4_Convergence} is devoted to give a proof for convergence of the proposed scheme
under some assumptions on the domain $\Omega$ and the contact angle function $\beta$.
{Finally, we conclude with summarizing consequence of this paper in Section \ref{sec:4_Conclusion}}.

{A preliminary version of this paper was included in the first author's phD thesis.}

\section{Preliminaries}\label{sec:4_Preliminaries}
\subsection{Convex analysis}\label{sec:4_ConvexAnalysis}
In this section, we recall basic facts from the convex analysis.
Throughout this section, let $X$ be a normed space and $X^*$ be the dual space of $X$.
Let $f:X\to\mathbb{R}\cup\{\pm\infty\}$.
\begin{dfn}[Subdifferential]
    For $u\in X$, the subdifferential $\subdiff{f}{u}\subset X^*$ is defined by
    \begin{equation*}
        p\in\subdiff{f}{u} :\Longleftrightarrow f(v) \geq f(u) + \innerproductSingle{p}{v - u}\qquad{\mbox{for all}}\quad v\in X.
    \end{equation*}
\end{dfn}
\begin{dfn}[Conjugate function]
    The conjugate function $f^*:X^*\to\mathbb{R}\cup\{\pm\infty\}$ of $f$ is defined by
    \begin{equation*}
        f^*(p) := \sup_{u\in X}\{\innerproductSingle{p}{u} - f(u)\}\qquad{\mbox{for}}\quad p\in X^*.
    \end{equation*}
\end{dfn}
\begin{prop}[Fenchel identity]\label{prop:FenchelIdentity}
    Suppose that $f$ is lower semi-continuous and convex. Then, it holds that
    {for $u\in X$ and $p\in X^*$,}
    \begin{equation*}
        p\in\subdiff{f}{u} \Longleftrightarrow u\in\subdiff{f^*}{p} \Longleftrightarrow f(u) + f^*(p) = \innerproductSingle{p}{u}
    \end{equation*}
\end{prop}
\begin{proof}
    See \cite[Proposition 5.1, Corollary 5.2]{EkelandTemam}.
\end{proof}
\begin{rem}[Characterization of conjugate functions]\label{rem:CharacterConjugate}
    Suppose that $f$ is positively homogeneous of degree $1$ and $f(0) = 0$. Then, it is easy to see that
    \begin{equation*}
        f^*(p) = 
        \begin{cases}
            0\ \ \mbox{if}\ \ p\in K,\\
            \infty\ \ \mbox{if}\ \ p\notin K,
        \end{cases}
    \end{equation*}
    where $K = \subdiff{f}{0}$.
\end{rem}
\subsection{Viscosity solution}\label{sec:ViscositySolution}
In this section, we briefly recall the notion of viscosity solution{s} which is a kind of weak solutions to degenerate parabolic equations. 
Partial differential equations under consideration {is of} the form:
\begin{equation}\label{eq:4_DegeneraateParabolicEquation}
    \begin{cases}
        u_t + F(x,t,u,\nabla u, \nabla^2 u) = 0\ \ \mbox{in}\ \ \Omega\times(0,T),\\
        B(x,t,u,\nabla u, \nabla^2 u) = 0\ \ \mbox{on}\ \ \pOmega\times(0,T),\\
        u(0,\cdot) = u_0\ \ \mbox{in}\ \ \closure{\Omega},
    \end{cases}
\end{equation}
where $F$ and $B$ are respectively functions defined on dense subset{s} in $\closure{\Omega}\times[0,T]\times\mathbb{R}\times\mathbb{R}^d\times\mathbb{S}^d$ 
and $\pOmega\times[0,T]\times\mathbb{R}\times\mathbb{R}^d\times\mathbb{S}^d$.
Here, $\mathbb{S}^d$ denotes the set of all symmetric matrices in $\mathbb{R}^{d\times d}$.
In this setting, let us define a {viscosity} sub- and supersolution to the problem $\eqref{eq:4_DegeneraateParabolicEquation}$.
\begin{dfn}[Viscosity solution]
    A function $u:\closure{\Omega}\times(0,T)\to\mathbb{R}$ is called a viscosity subsolution to $\eqref{eq:4_DegeneraateParabolicEquation}$ provided that $u^*(x,t)<\infty$ for all $(x,t)\in\closure{\Omega}\times(0,T)$ and for any $\varphi\in C^2(\closure{\Omega}\times(0,T))$ and $(\hat{x},\hat{t})\in\closure{\Omega}\times(0,T)$ such that $u^* - \varphi$ takes a local maximum at $(\hat{x},\hat{t})$,
    \begin{equation}\label{eq:4_F-solution}
        \begin{cases}
            \varphi_t(\hat{x},\hat{t}) + F_*(\hat{x},\hat{t},u^*(\hat{x},\hat{t}),\nabla\varphi(\hat{x},\hat{t}), \nabla^2\varphi(\hat{x},\hat{t})) \leq 0\ \ \mbox{if}\ \ \nabla\varphi(\hat{x},\hat{t}) \neq 0, \\
            \varphi_t(\hat{x},\hat{t}) \leq 0\ \ \mbox{if}\ \ \nabla\varphi(\hat{x},\hat{t})=0\ \ \mbox{and}\ \ \nabla^2\varphi(\hat{x},\hat{t})=O
        \end{cases}
    \end{equation}
    holds if $\hat{x}\in\Omega$ and either $\eqref{eq:4_F-solution}$ or $B_*(\hat{x},\hat{t},u^*(\hat{x},\hat{t}),\nabla\varphi(\hat{x},\hat{t}),\nabla^2\varphi(\hat{x},\hat{t})) \leq 0$ if $\hat{x}\in\pOmega$.
    {A function} $u$ is called a viscosity supersolution to $\eqref{eq:4_DegeneraateParabolicEquation}$ provided that $u_*(x,t)>-\infty$ for all $(x,t)\in\closure{\Omega}\times(0,T)$ and for any $\varphi\in C^2(\closure{\Omega}\times(0,T))$ and $(\hat{x},\hat{t})\in\closure{\Omega}\times(0,T)$ such that $u_*-\varphi$ takes a local minimum at $(\hat{x},\hat{t})$,
    \begin{equation}\label{eq:4_F-solution-super}
        \begin{cases}
            \varphi_t(\hat{x},\hat{t}) + F^*(\hat{x},\hat{t},u_*(\hat{x},\hat{t}),\nabla\varphi(\hat{x},\hat{t}), \nabla^2\varphi(\hat{x},\hat{t})) \geq 0\ \ \mbox{if}\ \ \nabla\varphi(\hat{x},\hat{t}) \neq 0, \\
            \varphi_t(\hat{x},\hat{t}) \geq 0\ \ \mbox{if}\ \ \nabla\varphi(\hat{x},\hat{t})=0\ \ \mbox{and}\ \ \nabla^2\varphi(\hat{x},\hat{t})=O
        \end{cases}
    \end{equation}
    holds if $\hat{x}\in\Omega$ and either $\eqref{eq:4_F-solution-super}$ or $B^*(\hat{x},\hat{t},u_*(\hat{x},\hat{t}),\nabla\varphi(\hat{x},\hat{t}),\nabla^2\varphi(\hat{x},\hat{t}))\geq 0$ if $\hat{x}\in\pOmega$.
    {A function} $u$ is called a viscosity solution if $u$ is a sub- and supersolution of $\eqref{eq:4_DegeneraateParabolicEquation}$.
\end{dfn}
\begin{dfn}[Degenerate ellipticity]\label{def:4_degenerate_elliptic}
    $F:{\closure{\Omega}}\times[0,T]{\times\mathbb{R}}\times\mathbb{R}^d\times\mathbb{S}^d\to\mathbb{R}$ is said to be degenerate elliptic if
    {for any $(x,t,r,p)\in\closure{\Omega}\times[0,T]\times\mathbb{R}\times(\mathbb{R}^d\backslash\{\zerovec\})$, it holds that}
    \begin{equation*}
        F(x,t,r,p,X) \leq F(x,t,r,p,Y)\ \ \mbox{for}\ {X,Y\in\mathbb{S}^d}\quad{\mbox{with}}\ \ X\geq Y,
    \end{equation*}
    where $X\leq Y$ means that $\left<X\xi,\xi\right>\leq\left<Y\xi,\xi\right>$ for every $\xi\in\mathbb{R}^d$.
\end{dfn}
Let us recall a comparison principle for the problem $\eqref{eq:4_DegeneraateParabolicEquation}$ {from \cite[Theorem 3.1]{Ba}:}
\begin{thm}[Comparison principle]\label{thm:4_ComparisonPrinciple}
    Let $\Omega\subset\mathbb{R}^d$ be a bounded domain {with $C^{2,1}$ boundary,} and {let} $u_0\in C(\closure{\Omega})$. 
    Let $u$ and $v$ be respectively a bounded upper semi-continuous subsolution and a bounded lower semi-continuous supersolution 
    of $\eqref{eq:4_DegeneraateParabolicEquation}$. Suppose that $F$ and $B$ {fulfill the following conditions:}
    {
        \begin{description}
            \item[(F1)] For every $R>0$, there exists a constant $C_R\in\mathbb{R}$ such that for every 
            $x\in\closure{\Omega}$, $t\in[0,T]$, $-R\leq v\leq u\leq R$, $p\in\mathbb{R}^d$ and $X\in\mathbb{S}^d$,
            it holds that
            \begin{equation*}
                F(x,t,u,p,X) - F(x,t,v,p,X)\geq C_R(u-v).
            \end{equation*}
            \item[(F2)] For any $R, K > 0$, there exists a function $\omega_{R,K}:[0,\infty)\to\mathbb{R}$ such that
            $\omega_{R,K}(s)\to 0$ as $s\downarrow 0$ and for all $\eta > 0$, it holds that
            \begin{equation*}
                F(y,t,u,q,Y) - F(x,t,u,p,X) \leq \omega_{R,K}\left(\eta + |x-y|(1 + |p|\lor |q|) + \frac{|x-y|^2}{\varepsilon^2}\right)
            \end{equation*}
            for any $x,y\in\closure{\Omega}$, $t\in[0,T]$, $u\in[-R,R]$, $p,q\in\mathbb{R}^d\backslash\{\zerovec\}$ and
            $X,Y\in\mathbb{S}^d$ satisfying
            \begin{align}
                &-\frac{K\eta}{\varepsilon^2}I_{2d}\leq\begin{pmatrix}
                    X & O\\
                    O & -Y
                \end{pmatrix}\leq\frac{K\eta}{\varepsilon^2}\begin{pmatrix}
                    I_d & -I_d\\
                    -I_d & I_d
                \end{pmatrix} + K\eta I_{2d},\label{eq:condition_F2} \\
                &|p - q| \leq K\varepsilon(|p|\land |q|)\qquad\mbox{and}\qquad|x - y| \leq K\eta\varepsilon,\nonumber
            \end{align}
            where $|p|\lor|q| :=\max\{|p|,|q|\}$ and $|p|\land|q|:=\min\{|p|,|q|\}$.
            \item[(F3)] $F\in C(\closure{\Omega}\times[0,T]\times\mathbb{R}\times(\mathbb{R}^d\backslash\{\zerovec\})\times\mathbb{S}^d)$ and $F^*(x,t,u,\zerovec,O) = F_*(x,t,u,\zerovec,O)$
            for every $x\in\closure{\Omega}$, $t\in[0,T]$ and $u\in\mathbb{R}$. In other words, $F(x,t,u,\cdot,\cdot)$ is continuous at $(\zerovec,O)$.
            \item[(B1)] For any $R>0$, there exists $C_R > 0$ such that for all $\lambda > 0$, $x\in\pOmega$, $t\in[0,T]$, $-R\leq v\leq u\leq R$ and $p\in\mathbb{R}^d$, it holds that
            \begin{equation*}
                B(x,t,u,p+\lambda\nu_\Omega(x)) - B(x,t,v, p) \geq C_R\lambda.
            \end{equation*}
            \item[(B2)] There exists a constant $C>0$ such that for any $x,y\in\closure{\Omega}$, $t\in[0,T]$, $u\in\mathbb{R}$ and $p,q\in\mathbb{R}^d$, it holds that
            \begin{equation*}
                |B(x,t,u,p) - B(y,t,u,q)| \leq C\{(|p| + |q|)|x-y| + |p-q|\}.
            \end{equation*}
        \end{description}
    }
    Then, it holds that $u\leq v$ in $\closure{\Omega}\times[0,T]$.
\end{thm}
{
    Under regularity assumption{s} on $\Omega$ and the contact angle function $\beta$, we confirm that the problem \eqref{eq:4_LevelSetEquationBdd} satisfies the comparison principle:
    \begin{thm}
        Suppose that $\Omega$ is uniformly $C^2$, thus $\nu_\Omega$ is uniformly $C^1$. 
        Assume that $\|\beta\|_\infty < 1$ and $\|\nabla_{\pOmega}\beta\|_\infty < \infty$.
        Then, the function $F$ and $B$ defined by \eqref{eq:intro_F_formula} and \eqref{eq:intro_B_formula} satisfy the hypotheses of \Theorem{thm:4_ComparisonPrinciple}.
        In particular, if $\Omega$ is $C^{2,1}$, then the comparison principle is available for the problem \eqref{eq:4_LevelSetEquationBdd}.
    \end{thm}
    \begin{proof}
        Since $F$ is independent of $u$, the condition (F1) is clearly fulfilled by setting $C_R=0$.
        To show the condition (F2), fix any $K,R,\eta,\varepsilon>0$ and let $X,Y\in\mathbb{S}^d$ be such that \eqref{eq:condition_F2}.
        Let $\{\bv e_1,\cdots,\bv e_d\}\subset\mathbb{R}^d$ be the standard basis of $\mathbb{R}^d$, i.e., for each $1\leq i\leq d$,
        the $j$-th element of $\bv e_i$ equals $\delta_{ij}$, where $\delta_{ij}$ denotes Kronecker's delta.
        Letting $\xi := (\bv e_i, \bv e_i)\in\mathbb{R}^{2d}$ and applying $\xi$ to \eqref{eq:condition_F2}, we have
        \begin{equation}\label{eq:proof_F2_1}
            -\frac{2k\eta}{\varepsilon^2}\leq -y_{ii} + x_{ii} \leq 2k\eta.
        \end{equation}
        Summing up \eqref{eq:proof_F2_1} through $1\leq i\leq d$, we obtain
        \begin{equation}\label{eq:proof_F2_2}
            -\frac{2K\eta d}{\varepsilon^2} \leq -\operatorname{tr}(Y) + \operatorname{tr}(X)\leq 2K\eta d.
        \end{equation}
        Meanwhile, letting $\xi := (\frac{p}{|p|},\frac{q}{|q|})\in\mathbb{R}^{2d}$ and applying $\xi$ to \eqref{eq:condition_F2} yield
        \begin{equation}\label{eq:proof_F2_3}
            -\frac{2K\eta}{\varepsilon^2}\leq\operatorname{tr}\left(\frac{p\otimes p}{|p|^2}X\right) - \operatorname{tr}\left(\frac{q\otimes q}{|q|^2}Y\right)
            \leq \frac{2K\eta}{\varepsilon^2}\left(1-\frac{\innerproductSingle{p}{q}}{|p||q|}\right) + 2K\eta.
        \end{equation}
        Combining \eqref{eq:proof_F2_2} and \eqref{eq:proof_F2_3} together with the Schwarz inequality gives
        \begin{equation*}
            -2K\left(1 + \frac{d+2}{\varepsilon^2}\right)\eta\leq F(q,Y) - F(p,X) \leq 2K\left(d + \frac{1}{\varepsilon^2}\right)\eta.
        \end{equation*}
        Thus, we define $\omega_{R,K}(s) := 2K\left(d+\frac{1}{\varepsilon^2}\right)s$ for each $s\in[0,\infty)$ and observe that
        $\lim_{s\downarrow 0}\omega_{R,K}(s) = 0$ and
        \begin{equation*}
            F(q,Y) - F(p,X) \leq \omega_{R,K}(\eta)\leq\omega_{R,K}\left(\eta + |x-y|(1+|p|\lor|q|) + \frac{|x-y|^2}{\varepsilon^2}\right).
        \end{equation*}
        Hence, $F$ satisfies (F2). For (F3), it is known that $F_*(\zerovec,O) = F^*(\zerovec,O) = 0$ by \cite[Lemma 1.6.16]{G}.
        Let us check that $B$ satisfies (B1). We compute
        \begin{align*}
            B(u,p+\lambda\nu_\Omega(x)) - B(u,p) 
            &= \innerproductSingle{p+\lambda\nu_\Omega(x)}{\nu_\Omega(x)} + \beta(x)|p + \lambda\nu_\Omega(x)| - \innerproductSingle{p}{\nu_\Omega(x)} - \beta(x)|p|\\
            &= \lambda + \beta(x)(|p+\lambda\nu_\Omega(x)| - |p|) \geq \lambda - \|\beta\|_\infty\left| |p + \lambda\nu_\Omega(x)| - |p| \right|\\
            &\geq \lambda - \|\beta\|_\infty|p + \lambda\nu_\Omega(x) - p| = (1-\|\beta\|_\infty)\lambda.
        \end{align*}
        Thus, letting $C_R := 1-\|\beta\|_\infty >0$, we have (B1).
        For (B2), we compute
        \begin{align}\label{eq:proof_B2_1}
            |B(x,p) - B(y,q)| &= |\innerproductSingle{p}{\nu_\Omega(x)} + \beta(x)|p| - \innerproductSingle{q}{\nu_\Omega(y)} -\beta(y)|q||\nonumber\\
            &\leq |\innerproductSingle{p}{\nu_\Omega(x)} - \innerproductSingle{q}{\nu_\Omega(y)}| + |\beta(x)|p| - \beta(y)|q||.
        \end{align}
        The first term of the right-hand side of \eqref{eq:proof_B2_1} can be estimated as follows.
        \begin{align}\label{eq:proof_B2_2}
            &|\innerproductSingle{p}{\nu_\Omega(x)} - \innerproductSingle{q}{\nu_\Omega(y)}| = |\innerproductSingle{p-q}{\nu_\Omega(x)} + \innerproductSingle{q}{\nu_\Omega(x)- \nu_\Omega(y)}|\nonumber\\
            &\leq |p-q| + |q| \|\nabla_{\pOmega}\nu_\Omega\|_\infty|x-y| \leq |p-q| + \|\nabla_{\pOmega}\nu_\Omega\|_\infty(|p| + |q|)|x-y|.
        \end{align}
        Whereas, for the second term of the right-hand side of \eqref{eq:proof_B2_1}, we compute
        \begin{align}\label{eq:proof_B2_3}
            &|\beta(x)|p| - \beta(y)|q|| \leq |\beta(x) - \beta(y)||p| + |\beta(y)|p| - \beta(y)|q||\nonumber\\
            &\leq \|\nabla_{\pOmega}\beta\|_\infty|p||x-y| + \|\beta\|_\infty||p| - |q|| \leq \|\nabla_{\pOmega}\beta\|_\infty(|p| + |q|)|x-y| + |p-q|.
        \end{align}
        Combining \eqref{eq:proof_B2_1}, \eqref{eq:proof_B2_2} and \eqref{eq:proof_B2_3} and setting $C := \max\{2,\|\nabla_{\pOmega}\beta\|_\infty,\|\nabla_{\pOmega}\nu_\Omega\|_\infty\} > 0$,
        we derive the desired inequality.
    \end{proof}
}



\section{Capillary Chambolle type scheme}\label{sec:4_CapillaryChambolleTypeScheme}
In this section, we will explore some properties of a minimizer of the energy functional $\chambolleEnergyFour{u}$ defined by
\begin{equation}\label{eq:4_chambolleEnergy}
    \chambolleEnergyFour{u} := C_{\beta}(u) + \frac{1}{2h}\int_{\Omega}(u - g)^2\,\dL{d},
\end{equation}
where $g\in L^2(\Omega)$ is a given data.


The following lemma asserts the monotonicity of the minimizer of $\chambolleEnergyFour{u}$
with respect to the data $g$:

\begin{lem}\label{lem:Monotonicity}
    Let $w_f$ denote the unique minimizer of \eqref{eq:4_chambolleEnergy}.
    Let $f,g\in L^2(\Omega)$ and suppose that $f\leq g$ holds a.e. in $\Omega$. Then, $w_f\leq w_g$ holds a.e. in $\Omega$.
\end{lem}
\begin{proof}
    Though the proof is quite similar to that of \cite{Chambolle2004}, we include it for the reader's convenience. 
    Our present purpose is to show {that} the set $\{w_f > w_g\}$ has zero $d$-dimensional Lebesgue measure. Since $w_f$ and $w_g$ are respectively 
    minimizers of \eqref{eq:4_chambolleEnergy} for the data $f$ and $g$, we have

    \begin{eqnarray}
        C_\beta(w_f) + \frac{1}{2h}\int_\Omega(w_f - f)^2\,\dL{d} &\leq& C_\beta(w_f\land w_g) + \frac{1}{2h}\int_\Omega(w_f\land w_g - f)^2\,\dL{d}, \\
        C_\beta(w_g) + \frac{1}{2h}\int_\Omega(w_g - g)^2\,\dL{d} &\leq& C_\beta(w_f\lor w_g) + \frac{1}{2h}\int_\Omega(w_f\lor w_g - g)^2\,\dL{d}.
    \end{eqnarray}
    We now recall a well-known inequality:
    \begin{equation}
        \int_\Omega|\nabla(u\land v)| + \int_\Omega|\nabla(u\lor v)| \leq \int_\Omega|\nabla u| + \int_\Omega|\nabla v|.
    \end{equation}
    Moreover, it is easily observed that
    \begin{equation}
        \int_{\partial\Omega}\beta\gamma(u\land v)\dH{d-1} + \int_{\partial\Omega}\beta\gamma(u\lor v)\dH{d-1} = \int_{\partial\Omega}\beta\gamma u\dH{d-1} + \int_{\partial\Omega}\beta\gamma v\dH{d-1}.
    \end{equation}
    Thus, we obtain
    \begin{align}\label{eq:4_MotonicityInequality}
        \begin{split}
            &\int_\Omega (w_f - f)^2\,\dL{d} \leq \int_\Omega (w_f\land w_g - f)^2\,\dL{d},\\
            &\int_\Omega (w_g - g)^2\,\dL{d} \leq \int_\Omega (w_f\lor w_g - g)^2\,\dL{d}.
        \end{split}
    \end{align}
    Splitting $\Omega$ into $\{w_f\leq w_g\}$ and $\{w_f > w_g\}$ and summing up $\eqref{eq:4_MotonicityInequality}$ yield
    \begin{align*}
        &\int_{\{w_f > w_g\}} (w_f - f)^2\,\dL{d}\leq \int_{\{w_f > w_g\}}(w_g - f)^2\,\dL{d},\\
        &\int_{\{w_f > w_g\}} (w_g - g)^2\,\dL{d}\leq \int_{\{w_f > w_g\}}(w_f - g)^2\,\dL{d}{.}
    \end{align*}
    {Adding two inequalities yield}
    \begin{equation}
        \int_{\{w_f > w_g\}}(w_f-w_g)(g-f)\,\dL{d}\leq 0.
    \end{equation}
    Since $f\leq g$ a.e. in $\Omega$, the integral domain $\{w_f > w_g\}$ should have zero $d$ dimensional measure.
\end{proof}

{We next prove that the map $L^2(\Omega)\ni g\mapsto\weaksol{h}{g}\in L^2(\Omega)$ is Lipschitz continuous:}

\begin{prop}\label{prop:Lipschitz_Continuity}
    For every $h>0$ and $f, g\in L^2(\Omega)$, it holds that
    \begin{equation*}
        \|\weaksol{h}{f} - \weaksol{h}{g}\|_2\leq \|f -g\|_2.
    \end{equation*}
\end{prop}
\begin{proof}
    Set $p_f := -(\weaksol{h}{f} - f) / h$ and $p_g := -(\weaksol{h}{g} - g) / h$. 
    Then, we see that $p_f\in\partial{C_\beta}(\weaksol{h}{f})$ and $p_g\in\partial{C_\beta}(\weaksol{h}{g})$.
    Hence, we get
    \begin{align}\label{eq:4_Lipschitz_Continuity_1}
        \begin{split}
            &{C_\beta}(\weaksol{h}{g}) \geq {\int_\Omega{p_f}{(\weaksol{h}{g} - \weaksol{h}{f})}\,\dL{d}} + {C_\beta}(\weaksol{h}{f}),\\
            &{C_\beta}(\weaksol{h}{f}) \geq {\int_\Omega{p_g}{(\weaksol{h}{f} - \weaksol{h}{g})}\,\dL{d}} + {C_\beta}(\weaksol{h}{g}).
        \end{split}
    \end{align}
    Summing up both sides of $\eqref{eq:4_Lipschitz_Continuity_1}$ gives
    \begin{equation}
        0\leq {\int_\Omega(p_f - p_g)(\weaksol{h}{f} - \weaksol{h}{g})\,\dL{d}} = {\int_\Omega\left(\frac{f-g-\weaksol{h}{f} + \weaksol{h}{g}}{h}\right)(\weaksol{h}{f} - \weaksol{h}{g})\,\dL{d}}.
    \end{equation}
    Thus, we obtain from Cauchy-Schwarz' inequality that
    \begin{equation*}
        \|\weaksol{h}{f} - \weaksol{h}{g}\|^2_2 \leq \|f-g\|_2\|\weaksol{h}{f} - \weaksol{h}{g}\|_2.
    \end{equation*}
    The proof is now complete.
\end{proof}

In the sequel, we investigate conditions on the domain $\Omega$ and on the contact angle function $\beta$
to ensure that a solution $\weaksol{h}{g}$ to \eqref{eq:4_Neumann_problem_master} is
equi-continuous with respect to $h>0$:
\begin{equation}\label{eq:4_Neumann_problem_master}
    \left\{
        \begin{array}{rcl}
            w + h\operatorname{div}\grad{\phi(\grad{w})} &=& g\ \ \mbox{in}\ \ \Omega,\\
            B({\cdot},\grad{w})&=& 0\ \ \mbox{on}\ \ \pOmega,
        \end{array}
    \right.
\end{equation}
where {$\anisotropy{}{}{p} := |p|$, and} $B(x,p) := \innerproductSingle{p}{\normalConv_\Omega{(x)}} + \beta(x)|p|$ {for $x\in\pOmega$ and $p\in\mathbb{R}^d$}.
Precisely speaking, we obtain the following result:

\begin{thm}\label{thm:4_grad_bound_of_weaksol}
    Suppose that $\Omega$ is a convex domain and $\beta\in C^1(\pOmega)$ with $\|\beta\|_\infty < 1$.
    Suppose that $w$ is a solution to \eqref{eq:4_Neumann_problem_master}.
    Assume that $\gradbdd{\pOmega}{\beta}(x)$ is orthogonal to the kernel of the Weingarten map {$\nabla_{\pOmega}\nu_\Omega(x)$} at each $x\in\pOmega$.
    Assume that $|\gradbdd{\pOmega}{\beta}(x)|\leq k(x)$ where $k(x)$ is the minimal {nonnegative} {principal (inward)} curvature of $\pOmega$ at $x\in\pOmega$.
    Then, it holds that 
    \begin{equation*}
        \|\grad{w}\|_\infty\leq \|\grad{g}\|_\infty.
    \end{equation*}
    In other words, if the gradient of $g$ is bounded, then 
    $\weaksol{h}{g}$ is equi-continuous with respect to $h>0$.
\end{thm}
{
\begin{rem}
    Our assumption on $\beta$ implies that $\beta$ must be a constant function if $\pOmega$ is flat.
    Precisely speaking, if $\pOmega$ is flat in the direction of $x_i$-axis, then $\beta$ must be independent of $x_i$.
\end{rem}
}
{
    \begin{rem}\label{rem:4_equi-continuity-g}
        We note that the function $w^h_g$ turns out to be equi-continuous with respect to $g$
        in the case when $g$ is the signed geodesic distance function.
    \end{rem}
}
To prove Theorem \ref{thm:4_grad_bound_of_weaksol}, we need several lemmas.
\begin{lem}\label{lem:4_comparison_principle}
    Let $L$ be a degenerate elliptic differential operator of the form:
    \begin{equation*}
        L := \sum_{i{,j}=1}^d a_{ij}\frac{\partial^2}{\partial x_i\partial x_j} + \sum_{\ell=1}^d b_\ell\frac{\partial}{\partial x_\ell}
    \end{equation*}
    with a non-positive definite symmetric matrix $(a_{ij})_{{1\leq i,j\leq d}}$ and $b_\ell$, where $(a_{ij})^{1/2}$ is Lipschitz 
    and $b_\ell$'s are uniformly continuous in $\closure{\Omega}$ where $\Omega$ is a domain in $\mathbb{R}^d$.
    Assume that $\partial\Omega$ is uniformly $C^1$.
    Let {$\xi$} be a bounded $C^1$ vector field on $\partial\Omega$ such that $\inf_{\partial\Omega}{\innerproductSingle{{\xi}}{\normalConv_\Omega}}>0$ where $\normalConv_\Omega$ is the exterior unit normal vector field of $\partial\Omega$.
    If $v\in C^2(\Omega)\cap C^1(\closure{\Omega})$ satisfies
    \begin{equation}\label{eq:4_strong_maximum_principle_pde}
    \left\{
    \begin{array}{rlcc}
    	v + Lv &\hspace{-7pt}\leq \lambda &\text{in}& \Omega{,} \\
    	\innerproductSingle{{\xi}}{\grad{v}} &\hspace{-7pt}\leq 0 &\text{on}& \partial\Omega
    \end{array}
    \right.
    \end{equation}
    for some constant $\lambda${,}
    {then} it holds that $v\leq\lambda$ in $\Omega$.
\end{lem}
\begin{proof}
    Since $v\equiv \lambda$ is a solution, the assertion follows from a {classical} comparison principle {for linear equations (see e.g. \cite{ProtterWeinberger1984})}.
\end{proof}
\begin{rem}
    The comparison principle for the oblique derivative boundary problem is well known even for a viscosity solution
    as stated in Theorem \ref{thm:4_ComparisonPrinciple}.
    See e.g. \cite{Ba, IshiiLions1990}.
\end{rem}

\begin{lem}\label{lem:4_sufficient_cond_boundary}
Assume that $\partial\Omega$ is uniformly $C^2$ so that $\normalConv_\Omega$ is uniformly $C^1$.
 Assume that $w{\in C^3(\Omega)}$ is $C^2$ up to the boundary.
 Assume that $w$ satisfies $B({\cdot},\nabla w{(\cdot)})=0$ on $\partial\Omega$.
 Assume that $\beta$ is $C^1$, $\|\beta\|_\infty<1$ and $\|\nabla_{\partial\Omega}\beta\|_\infty<\infty$.
 If $B$ satisfies
\[
	\sum_{{i}=1}^d w_{{i}}(x) \frac{\partial B}{\partial x_{{i}}} (x,\nabla w{(x)}) \geq 0 \quad\text{for}\quad x\in\pOmega,
\]
then for $u:={\frac{1}{2}|\grad{w}|^2}$, {the $C^1$ vector field ${\xi}=\nabla_pB(\cdot,\nabla w(\cdot))$ satisfies}
\[
	{\innerproductSingle{{\xi}}{\grad{u}}} \leq 0 \quad\text{on}\quad \partial\Omega
\]
and
\[
    {\innerproductSingle{{\xi}}{\nu_{\Omega}}}\geq 1-\|\beta\|_\infty>0 \quad\text{on}\quad\pOmega.
\]
\end{lem}
\begin{proof}
    First differentiate $B(x,\nabla w{(x)})=0$ in $x_i$ and multiply $w_i{(x)}$ to get
    \begin{equation}\label{eq:4_lemma3_6_1}
    	\sum_{\ell=1}^d w_i{(x)} \frac{\partial B}{\partial p_\ell}{(x,\grad{w(x)})} w_{i\ell}{(x)} + w_i{(x)} \frac{\partial B}{\partial x_i}{(x,\grad{w(x)})} = 0{.}
    \end{equation}
    We sum up \eqref{eq:4_lemma3_6_1} from $i=1$ to $d$ to get 
    \[
    	\sum_{\ell=1}^d \frac{\partial B}{\partial p_\ell}{(x,\grad{w(x)})} \frac{\partial u}{\partial x_\ell}{(x)}
    	+ \sum_{i=1}^d w_i{(x)} \frac{\partial B}{\partial x_i}{(x,\grad{w(x)})} = 0.
    \]
    By our assumption, we now obtain
    \[
        {
            \innerproductSingle{{\xi}(x)}{\grad{u(x)}} = \nabla_{p}B(x,\grad{w(x)})\cdot\grad{u(x)} = \sum_{\ell=1}^d\frac{\partial B}{\partial p_\ell}(x,\grad{w(x)})\frac{\partial u}{\partial x_\ell}(x)\leq 0.
        }
    \]
    {Meanwhile, a} direct calculation shows
    \begin{equation*}
        \frac{\partial B}{\partial p_\ell}{(x,p)} = \nu_\ell(x) + \beta(x) \frac{p_\ell}{|p|},
    \end{equation*}
    where $\normalConv_{\ell}$ denotes the $\ell$-th element of $\normalConv_\Omega$. We deduce from the Schwarz inequality that
    \begin{equation*}
        {\innerproductSingle{{\xi}(x)}{\normalConv_\Omega(x)}} = \sum_{\ell=1}^d\left(\normalConv_\ell(x){^2} + \beta(x)\frac{w_\ell(x){\nu_\ell(x)}}{|\grad{w(x)}|}\right) =  1 + \beta(x) \frac{{\innerproductSingle{\nabla w(x)}{\nu_{\Omega}(x)}}}{|\nabla w(x)|} \geq 1 - \|\beta\|_\infty.
    \end{equation*}
    {We now complete the proof.}
\end{proof}

\begin{lem}\label{lem:4_sufficient_cond_boundary_2}
    Assume that $\partial\Omega$ is uniformly $C^2$.
    Assume that $\nabla_{\partial\Omega}\beta(x)$ is orthogonal to the kernel of 
    the inward Weingarten map $\nabla_{\partial\Omega}\normalConv_\Omega(x)$ for each $x\in\pOmega$.
    If $|\nabla_{\partial\Omega}\beta(x)|$ is bounded 
    by the minimal {nonnegative} principal {(inward)} curvature {$\kappa(x)$} of $\partial\Omega$ at each $x\in\pOmega$, then
    \[
    \sum_{i=1}^d w_i(x) \frac{\partial B}{\partial x_i}\left(x,\nabla w(x)\right) \geq 0
    \]
    is fulfilled for every $x\in\pOmega$.
\end{lem}
\begin{proof}
    Since
    \[
    \frac{\partial B}{\partial x_{{i}}}{(x,\grad{w(x)})} = {\innerproductSingle{\frac{\partial\nu_{{\Omega}}}{\partial x_{{i}}}(x)}{\nabla w(x)}} + \frac{\partial\beta}{\partial x_{{i}}} (x) |\grad{w{(x)}}|,
    \]
    we see that
    \[
    \sum_{{i}=1}^d w_{{i}}(x) \frac{\partial B}{\partial x_{{i}}} (x{,\grad{w(x)}})
    = \sum_{\ell,{i}=1}^d \frac{\partial\nu_{{\ell}}}{\partial x_{{i}}}{(x)} w_\ell{(x)} w_{{i}}{(x)} + {\sum_{i=1}^d} w_{{i}}{(x)} \frac{\partial\beta}{\partial x_{{i}}}(x)|\nabla w{(x)}|.
    \]
    We extend $\beta$ as constant to the $\nu_{{\Omega}}$-direction.
     Since $\nabla_{\partial\Omega}\beta$ is orthogonal to the kernel of the Weingarten map {$\nabla_{\pOmega}\nu_\Omega$}, we proceed
    \begin{align*}
        \sum_{\ell,{i}=1}^d \frac{\partial\nu_{{\ell}}}{\partial x_{{i}}}{(x)} w_\ell{(x)} w_{{i}}{(x)} + \sum_{{i}=1}^d w_{{i}}{(x)} \frac{\partial\beta}{\partial x_{{i}}}{(x)} |\nabla w{(x)}|
        &\geq \kappa(x) |\nabla w{(x)}|^2 - |\nabla w{(x)}|^2 |\nabla_{\partial\Omega}\beta{(x)}|\\
        &{=(\kappa(x) - |\nabla_{\pOmega}\beta(x)|) |\grad{w(x)}|^2\geq 0}.
    \end{align*}
    Our assumption guarantees that the right-hand side is positive.
    The proof is now complete.
\end{proof}

We are now in the position to prove Proposition \ref{thm:4_grad_bound_of_weaksol}.

\begin{proof}[Proof of Proposition \ref{thm:4_grad_bound_of_weaksol}]
    We define $u := \frac{1}{2}|\nabla w|^2$ and argue by Bernstein's method (see e.g. \cite[Chapter 15]{GilbargTrudinger1983}).
    We differentiate the first equation of \eqref{eq:4_Neumann_problem_master} in the direction $x_k\ (1\leq k\leq d)$ and {multiply it} by $w_k{(x)}$ {to get}
    \begin{equation}\label{eq:4_bernstein_1}
        w_k{(x)}^2 - h\sum_{i=1}^dw_k{(x)}\del{i}{\left(\sum_{j=1}^d\anisotropy{i}{j}{\grad{w{(x)}}}w_{jk}{(x)}\right)} = w_k{(x)}g_k{(x)}.
    \end{equation}
    Meanwhile, we calculate
    \begin{align}\label{eq:4_bernstein_2}
        &\sum_{i=1}^dw_k{(x)}\del{i}{\left(\sum_{j=1}^d\anisotropy{i}{j}{\grad{w{(x)}}}w_{jk}{(x)}\right)}\\\nonumber
        =& \sum_{i=1}^d\del{i}{\left(\sum_{j=1}^dw_k{(x)}\anisotropy{i}{j}{\grad{w{(x)}}}w_{jk}{(x)}\right)} - \sum_{i,j=1}^d w_{ik}{(x)}\anisotropy{i}{j}{\grad{w{(x)}}}w_{jk}{(x)}\\\nonumber
        \leq &\sum_{i=1}^d\del{i}{\left(\sum_{j=1}^d\anisotropy{i}{j}{\grad{w{(x)}}}w_k{(x)}w_{jk}{(x)}\right)} =\sum_{i=1}^d\del{i}{\left\{\sum_{j=1}^d\anisotropy{i}{j}{\grad{w{(x)}}}{\del{j}{}}\left(\frac{w_k{(x)}^2}{2}\right)\right\}}.
    \end{align}
    Here, the inequality in $\eqref{eq:4_bernstein_2}$ follows from the positive definiteness of\linebreak $(\anisotropy{i}{j}{\grad{w{(x)}}})_{1\leq i,j\leq d}$.
    Summing up $\eqref{eq:4_bernstein_1}$ over $1\leq k\leq d$ and taking $\eqref{eq:4_bernstein_2}$ into account,
    we obtain
    \begin{align*}
        &2u{(x)} - h\sum_{i=1}^d\del{i}{\left(\sum_{j=1}^d\anisotropy{i}{j}{\grad{w{(x)}}}u_j{(x)}\right)}\\
        & \leq \sum_{k=1}^d\left\{w_k{(x)}^2 - h\sum_{i=1}^dw_k{(x)}\del{i}{\left(\sum_{j=1}^d\anisotropy{i}{j}{\grad{w{(x)}}}w_{jk}{(x)}\right)}\right\}\\ 
        & {=}\sum_{k=1}^dw_k{(x)}g_k{(x)} \leq \frac{1}{2}|\grad{w{(x)}}|^2 + \frac{1}{2}|\grad{g{(x)}}|^2 \leq u{(x)} + {\frac{1}{2}\|\grad{g}\|_\infty^2}.
    \end{align*}
    Hence, $u$ satisfies the first equation of \eqref{eq:4_strong_maximum_principle_pde} 
    in the case where $\lambda := {\frac{1}{2}\|\grad{g}\|_\infty^2}$ and $L(u) = -\operatorname{div}(A(x)\grad{u})$ with $A(x) := (\anisotropy{i}{j}{\grad{w(x)}})_{1\leq i,j\leq d}$.
    We have already seen that $\innerproductSingle{{\xi}}{\grad{u}}\leq 0$ with ${\xi} := \gradbdd{p}{B({\cdot},\grad{w{(\cdot)}})}$ 
    in Lemma \ref{lem:4_sufficient_cond_boundary} and Lemma \ref{lem:4_sufficient_cond_boundary_2}.
    Therefore, we conclude from Lemma \ref{lem:4_comparison_principle} that 
    $\|\grad{w}\|_\infty\leq \|\grad{g}\|_\infty$.
\end{proof}

\begin{rem}
    To guarantee that $w$ is $C^2$ up to the boundary, we need a regularity assumption on $\Omega$ which is slightly more than $C^2$,
    say $C^{2,\alpha}$. {The} $C^{2,1}$ {assumption} is {sufficient} to {guarantee $C^2$ regularity for $w$}
    {(see e.g., \cite[\S 6.4]{GilbargTrudinger1983})}.
\end{rem}


We are now in the position to define a set operator $\chambolleSet{h}:\mathcal{P}(\closure{\Omega})\to\mathcal{P}(\closure{\Omega})$
where $\mathcal{P}(\closure{\Omega})$ denotes the set of all subsets in $\closure{\Omega}$.
For each $E\subset\closure{\Omega}$, we set
\begin{equation*}
    \chambolleSet{h}(E) := \{\weaksol{h}{E}\leq 0\},
\end{equation*}
where $\weaksol{h}{E}$ is the unique minimizer of $\chambolleEnergyFour{u}$.
Then, we have two important properties of the set operator $\chambolleSet{h}$ as follows.

\begin{lem}[Monotonicity of $\chambolleSet{h}$]\label{lem:4_ChambollesetMonotonicity}
    Suppose that $E\subset F\subset\closure{\Omega}$. 
    Then, it holds that $\chambolleSet{h}(E)\subset\chambolleSet{h}(F)$.
\end{lem}
\begin{proof}
    Since $\geodis{E}\geq \geodis{F}$ in $\closure{\Omega}$, we deduce from \Lemma{lem:Monotonicity} that
    $\weaksol{h}{E}\geq \weaksol{h}{F}$ in $\closure{\Omega}$. This readily yields $\chambolleSet{h}(E)\subset\chambolleSet{h}(F)$.
\end{proof}

\begin{rem}
    It is crucial to use the geodesic signed distance function $\geodis{E}$ as an initial data $g$
    {for} the variational problem \eqref{eq:4_chambolleEnergy}. Thanks to the monotonicity of $\geodis{E}$
    with respect to $E$ (see {e.g. \cite[Lemma 1]{EtoGiga2023}}), we do not need to assume the convexity of $\Omega$
    to obtain the monotonicity of $\chambolleSet{h}(E)$.
\end{rem}

{We are now in the position to establish the continuity of the scheme $\chambolleSet{h}$:}

\begin{lem}[Continuity of $\chambolleSet{h}$]\label{lem:ChambollesetContinuity}
    Let $\{E_n\}_n$ be a non-increasing sequence of relatively closed subsets in ${\closure{\Omega}}$.
    Then, it holds that
    \begin{equation*}
        \chambolleSet{h}\left(\bigcap_{n=1}^\infty E_n\right) = \bigcap_{n=1}^\infty\chambolleSet{h}(E_n).
    \end{equation*}
\end{lem}
\begin{proof}
    Let $E := \bigcap_{n=1}^\infty E_n$.
    Since $\bigcap_{n=1}^\infty\chambolleSet{h}(E_n)\supset\chambolleSet{h}(E)$ is obvious,
    we prove the converse inclusion. Suppose that $x\notin\chambolleSet{h}(E)$. Then, we have $\weaksol{h}{E}(x) > 0$.
    Since $\geodis{E_n}\to\geodis{E}$ pointwise as $n\to\infty$ and $\Omega$ is bounded, 
    this convergence is uniform in $\closure{\Omega}$. Hence, we deduce from Proposition \ref{prop:Lipschitz_Continuity} that
    $\weaksol{h}{E_n}\to\weaksol{h}{E}$ a.e. in $\Omega$ as $n\to\infty$.
    {Since $|\grad{\geodis{E_n}}|$ is uniformly bounded with respect to $n\in\mathbb{N}$},
    we see that $\{\weaksol{h}{E_n}\}_n$ is equi-continuous
    by {\Theorem{thm:4_grad_bound_of_weaksol} and \Remark{rem:4_equi-continuity-g}}. {Moreover,} the maximum principle {guarantees that $\weaksol{h}{E_n}$ is uniformly bounded with respect to $n\in\mathbb{N}$}.
    Hence, we can extract a subsequence $\{\weaksol{h}{E_{n_k}}\}_k$ by {the} Ascoli{--}Arzel\`a theorem
    {such that $\weaksol{h}{E_{n_k}}\to w$ uniformly in $\closure{\Omega}$ for some $w\in UC(\closure{\Omega})$. This $w$ should correspond to $\weaksol{h}{E}$. Letting $k\to\infty$,}
    we derive $\weaksol{h}{E_{n_k}}(x)\to\weaksol{h}{E}(x)$ which implies
    $\weaksol{h}{E_{n_k}}(x) > 0$ holds for sufficiently large $k\in\mathbb{N}$. This leads to $x\notin\bigcap_{n=1}^\infty\chambolleSet{h}(E_{n})$.
\end{proof}

As a conclusion of this section, we characterize the unique minimizer of \eqref{eq:4_chambolleEnergy} as the projection
to the data function onto a closed convex set in $L^2(\Omega)$:

\begin{prop}\label{prop:4_dual_argument_1}
    Let $g\in L^2(\Omega)$. Then, there exists $\overline{z}\in\mathbf{X}_2(\Omega)$ such that
    \begin{equation}\label{eq:4_DualMinimizing}
        \overline{z} = \operatorname{argmin}\left\{ \| \operatorname{div}{z} - g\|_2\ \biggm|\  \substack{z\in\mathbf{X}_2(\Omega),\ \|z\|_\infty\leq 1,\\ [z\cdot\nu] = -\beta\ \mathcal{H}^{d-1}\mbox{\small{-a.e.}}\ \mbox{\small{on}}\ \ \pOmega}\right\}.
    \end{equation}
\end{prop}
\begin{proof}
    We take a minimizing sequence $\{z_i\}_i\subset\mathbf{X}_2(\Omega)$ of $\eqref{eq:4_DualMinimizing}$.
    Since $\{z_i\}_i$ is bounded in $L^\infty(\Omega;\mathbb{R}^d)$, up to a subsequence, there exists $\overline{z}\in L^\infty(\Omega;\mathbb{R}^d)$ such that 
    \begin{equation*}
        z_i\rightharpoonup\overline{z}\quad \mbox{weakly-}*\quad \mbox{in}\quad L^\infty(\Omega;\mathbb{R}^d).
    \end{equation*}
    Meanwhile, since $\{\operatorname{div}{z_i}\}_i$ is bounded in $L^2(\Omega)$,
    there exists $z_{div}\in L^2(\Omega)$ such that 
    \begin{equation*}
        \operatorname{div}{z_i}\rightharpoonup z_{div}\quad \mbox{weakly}\quad\mbox{in}\quad L^2(\Omega)
    \end{equation*}
    by taking a further subsequence.
    Then, we have $\operatorname{div}{\overline{z}} = z_{div}$ in $\mathcal{D}'(\Omega)$. Indeed, for any $\varphi\in C^\infty_0(\Omega)$, we deduce
    \begin{align}\label{eq:4_prop1_1}
        \int_\Omega z_{div}\varphi\,\dL{d} &= \lim_{i\to\infty}\int_\Omega (\operatorname{div}{z_i})\,\varphi\,\dL{d}\nonumber\\ 
        &= \lim_{i\to\infty}\int_\Omega -z_i\cdot\nabla\varphi\,\dL{d} = -\int_\Omega\overline{z}\cdot\nabla\varphi\,\dL{d}.
    \end{align}
    Here, the second equality is deduced from \cite[Proposition C.4]{Mazon}. Thus, we see $\overline{z}\in\mathbf{X}_2(\Omega)$.
    Moreover, the lower semi-continuity of $\{z_i\}_i$ in the topology of weakly-$*$ $L^\infty(\Omega;\mathbb{R}^d)$ yields
    \begin{equation*}
        \|\overline{z}\|_\infty \leq \liminf_{i\to\infty}\|z_i\|_\infty\leq 1.
    \end{equation*}
    For any $\varphi\in C^\infty(\closure{\Omega})$, we obtain
    \begin{align}\label{eq:4_prop1_2}
        \int_\Omega z_{div}\,\varphi\,\dL{d} &= \lim_{i\to\infty}\int_\Omega(\operatorname{div}{z_i})\,\varphi\,\dL{d}\nonumber\\
        &= \lim_{i\to\infty}\left(\int_{\pOmega}[z_i\cdot\nu]\,\varphi\,\dH{d-1} - \int_\Omega z_i\cdot\nabla\varphi\,\dL{d}\right)\nonumber\\
        &= \lim_{i\to\infty}\left(\int_{\pOmega}-\beta\,\varphi\,\dH{d-1} - \int_\Omega z_i\cdot\nabla\varphi\,\dL{d}\right)\nonumber\\ 
        &= \int_{\pOmega}-\beta\,\varphi\,\dH{d-1} - \int_\Omega\overline{z}\cdot\nabla\varphi\,\dL{d}.
    \end{align}
    The left-hand side of $\eqref{eq:4_prop1_2}$ is also deformed as follows:
    \begin{equation}\label{eq:4_prop1_3}
        \int_\Omega z_{div}\,\varphi\,\dL{d} = \int_\Omega(\operatorname{div}{\overline{z}})\,\varphi\,\dL{d} = \int_{\pOmega}[\overline{z}\cdot\nu]\,\varphi\,\dH{d-1} - \int_\Omega \overline{z}\cdot\nabla\varphi\,\dL{d}.
    \end{equation}
    Combining $\eqref{eq:4_prop1_2}$ and $\eqref{eq:4_prop1_3}$ gives
    \begin{equation*}
        \int_{\pOmega}[\overline{z}\cdot\nu]\,\varphi\,\dH{d-1} = \int_{\pOmega}-\beta\,\varphi\,\dH{d-1}.
    \end{equation*}
    Since $\varphi\in C^\infty(\closure{\Omega})$ is arbitrary, we see that $[\overline{z}\cdot\nu] = -\beta$ holds $\mathcal{H}^{d-1}$-a.e. on $\pOmega$.
    Therefore, we conclude that $\overline{z}$ is a minimizer of $\eqref{eq:4_DualMinimizing}$.
\end{proof}

\begin{prop}\label{prop:4_dual_argument_2}
    Let $g\in L^2(\Omega)$. Suppose that $w$ is a solution of
    \begin{equation}\label{eq:4_EulerLagrange}
        \frac{u-g}{h} + \subdiff{C_\beta}{u} \ni 0.
    \end{equation}
    Then, $w = g - \projection{hK_\beta}{g}$ holds where $K_\beta = \subdiff{C_\beta}{0}$ and $\projectionnoparam{hK_\beta}$ is the orthogonal projection to the set $hK_\beta$ in $L^2(\Omega)$.
\end{prop}
\begin{proof}
    The variational problem $\eqref{eq:4_EulerLagrange}$ is equivalent to $w\in\subdiff{C_\beta^*}{(g - w)/h}$ due to Proposition \ref{prop:FenchelIdentity}.
    Setting $\overline{w} := (g - w) / h$, the problem is rewritten as 
    \begin{equation}\label{eq:4_EulerLagrangeDual}
        0\in h\overline{w} - g + \subdiff{C_\beta^*}{\overline{w}}.
    \end{equation}
    This implies that
    \begin{equation*}
        \overline{w} = \argmin_{\overline{u} \in L^2(\Omega)}{\left\{\frac{\|h\overline{u} - g\|^2_2}{2} + C_\beta^*(\overline{u})\right\}}.
    \end{equation*}
    Due to Remark \ref{rem:CharacterConjugate}, the above problem is reduced to
    \begin{equation}\label{eq:4_prop2_1}
        \overline{w} = \argmin_{\overline{u}\in K_\beta}{\norm{h\overline{u} - g}_2}{.}
    \end{equation}
    The formula $\eqref{eq:4_prop2_1}$ implies that 
    \begin{equation*}
        \overline{w} = \argmin_{\overline{u}\in K_\beta}\norm{h\overline{u}-g}_2 = \frac{1}{h}\argmin_{u\in hK_\beta}\norm{u-g}_2 = \frac{1}{h}\projection{hK_\beta}{g}.
    \end{equation*}
    Recalling $\overline{w} = (g-w)/h$, we conclude that $w = g - \projection{hK_\beta}{g}$.

\end{proof}
\begin{rem}
    {S}imilar arguments {as in} Proposition \ref{prop:4_dual_argument_1} and Proposition \ref{prop:4_dual_argument_2} can be found in \cite[\S 3]{Chambolle2004_TV}.
    Therein, Chambolle considered the dual problem (\eqref{eq:4_EulerLagrangeDual} in our case) instead of the original one (\eqref{eq:4_EulerLagrange} in our case) 
    to establish a gradient descent algorithm to compute a time discrete solution to the mean curvature flow.
\end{rem}

\section{Convergence of the proposed scheme}\label{sec:4_Convergence}
In this section, we shall show convergence of the approximate scheme $\approxscheme{}{h}$ to \eqref{eq:4_LevelSetEquationBdd}.
For this purpose, it is crucial to confirm that $\approxscheme{}{h}$ fulfills the conditions from 
\eqref{eq:4_Monotonicity} to \eqref{eq:4_consistency_sub}.

    We now define an approximate scheme $\approxscheme{}{h}$ for \eqref{eq:4_LevelSetEquationBdd} 
    with the aid of the capillary Chambolle type scheme $\chambolleSet{h}$ as follows:

    \begin{equation}\label{eq:4_S_h_def}
        {\chambolleFunction{h}{u_0}}(x) := \sup\{\lambda\in\mathbb{R}\mid x\in\chambolleSet{h}(\{u_0\geq \lambda\})\}.
    \end{equation}
    In terms of $\approxscheme{t}{h}$, we define an approximate solution $u^h$ to \eqref{eq:4_LevelSetEquationBdd} by
    \begin{equation*}
        u^h(t,x) := \approxscheme{\left[\frac{t}{h}\right]}{h}^{\lfloor\frac{t}{h}\rfloor}u_0(x).
    \end{equation*}
    {We now state a main result of this section as follows:}

\begin{thm}\label{thm:4_S_h_satisfies_MTC}
    Suppose that $\Omega$ is convex. Assume that there exist constants $\betalow < 0$ and $\betahigh > 0$ such that
    $-1 < \betalow \leq \beta\leq \betahigh < 1$ on $\pOmega$. Moreover, assume that $\beta$ satisfies the hypotheses of Theorem \ref{thm:4_grad_bound_of_weaksol}.
    Then, the approximate scheme $\chambolleFunction{h}{}$ defined by \eqref{eq:4_S_h_def} satisfies 
    the conditions from \eqref{eq:4_Monotonicity} to \eqref{eq:4_consistency_sub}.
\end{thm}

Once Theorem \ref{thm:4_S_h_satisfies_MTC} is established, 
the statement of Theorem \ref{thm:final_statement} immediately 
follows from the result by Barles and Souganidis:
\begin{thm}[\cite{BarlesSouganidis1991}, Theorem 2.1]
    Suppose that $F:(\mathbb{R}^d\backslash\{\zerovec\})\times\mathbb{S}^d\to\mathbb{R}$ 
    is degenerate elliptic, geometric, continuous, and satisfies 
    $-\infty < F_*(\zerovec,O) = F^*(\zerovec,O) < \infty$.
    Assume that the approximate scheme $\approxscheme{t}{h}$ fulfills the conditions 
    from $\eqref{eq:4_Monotonicity}$ to $\eqref{eq:4_consistency_sub}$.
    Then, $u^h$ {uniformly} converges to the unique viscosity solution of $\eqref{eq:4_LevelSetEquationBdd}$.
\end{thm}

We begin with confirmation that $\chambolleFunction{h}$ is monotone and translation invariant:

\begin{prop}[Monotonicity and translation invariance of $\chambolleFunction{h}$]\label{prop:4_Sh_monotone_translation}
    The function operator $\chambolleFunction{h}$ satisfies the criteria \eqref{eq:4_Monotonicity} and \eqref{eq:4_TranslationInvariance}.
\end{prop}
\begin{proof}
    Suppose that $u\leq v$ in $\closure{\Omega}$. Then, we see that $\{u\geq \lambda\}\subset\{v\geq \lambda\}$ for every $\lambda\in\mathbb{R}$.
    Since $\chambolleSet{h}$ is monotone from \Lemma{lem:4_ChambollesetMonotonicity}, 
    we have $\chambolleSet{h}(\{u\geq\lambda\})\subset\chambolleSet{h}(\{v\geq \lambda\})$. Thus, it follows from the definition of $\chambolleFunction{h}$ that
    $\chambolleFunction{h}u\leq\chambolleFunction{h}v$ in $\closure{\Omega}$. The formula $\chambolleFunction{h}(u+c) = \chambolleFunction{h}u + c$ is straightforward {by the definition of $S_h$}.
\end{proof}

The following lemma states a relationship between $\chambolleFunction{h}$ and $\chambolleSet{h}$ which will be crucial for our study.

\begin{lem}\label{lem:4_Sh_Th}
    For every $\lambda\in\mathbb{R}^2$ and $u_0\in UC({\closure{\Omega}})$, 
    ${\chambolleFunction{h}{u_0}}(x)\geq \lambda$ holds if and only if $x\in\chambolleSet{h}(\{u_0\geq \lambda\})$.
\end{lem}
\begin{proof}
    The assertion is straightforward by Lemma \ref{lem:Monotonicity}, Lemma \ref{lem:ChambollesetContinuity} and \cite[Lemma 4.3]{EtoGigaIshii2012}.
\end{proof}

We need a {soliton-like} rigorous solution to the mean curvature flow with the constant contact angle condition to capture a general flow.
A candidate for such a solution is a translative soliton:

\begin{dfn}[Translative soliton]\label{def:4_TS}
    A function $f:\closure{\Omega}\to\mathbb{R}$ is called a translative soliton
    if there exist a constant $k\in(-1,1)$ and a function $\solitonprofile{k}{x'}$ on some subset $\tilde{\Omega}\subset\mathbb{R}^{d-1}$
    such that $f(x',x_d) = \solitonprofile{k}{x'}- x_d$ holds and $f$ solves the following partial differential equation:
    \begin{equation}\label{eq:4_NeumannProblem}
        \begin{cases}
            w - h\operatorname{div}{\grad{\anisotropy{}{}{\grad{w}}}} = \geodis{F}\ \ \mbox{in}\ \ \Omega, \\
            \innerproductSingle{\grad{w}}{\normalConv_\Omega} + k|\grad{w}| = 0\ \ \mbox{on}\ \ \pOmega,
        \end{cases}
    \end{equation}
    where $F := \{(x',x_d)\mid x_d \leq\solitonprofile{k}{x'} + (\arctan{\alpha})h\}$ with 
    $\alpha := \arccos{k}$.
\end{dfn}

\begin{rem}
    A translative soliton is often called either a translator or a translating soliton in the literature, and
    it originally means a rigorous solution to the mean curvature flow with a contact angle condition 
    which can be represented as a graph over an ambient space.
    We note that the problem $\eqref{eq:4_NeumannProblem}$ is a discrete variant of {level-set} equations for the mean curvature flow.
    The translative soliton which we treat here corresponds to a bowl soliton which is restricted to a cylindrical domain.
    For a summary of existing works for translators, see e.g. \cite[\S 4]{HoffmanIlmanenMartinWhite2021}.
\end{rem}

To prove Theorem \ref{thm:4_S_h_satisfies_MTC}, we need an assumption and several lemmas:

\begin{lem}\label{lem:4_existence_ts}
    For every point $z\in\Omega$ and a vector $\mathbf{v}$, there exists a translative soliton which evolves
    in the direction either $\mathbf{v}$ or $-\mathbf{v}$ and includes $z$ in its level set.
\end{lem}
\begin{proof}
    This is straightforward from the result by Zhou \cite[Collorary 4.2]{Zhou2018}.
\end{proof}

If $\Omega$ is a smooth bounded domain, then the above lemma can be proved by approximating $\Omega$ with a cylindrical domain.
We can compute an exact form of $\solitonprofile{\beta}{x'}$ 
if $d = 2$ and $\Omega$ is a cylindrical domain as follows:

\begin{lem}
    For each $\alpha > 0$, we define a function $u_\alpha:[-1,1]\times[0,T]\to\mathbb{R}$ by
    \begin{equation}\label{eq:4_TS}
        u_\alpha(x,t) := \frac{1}{\arctan{\alpha}}\log{\left|\cos{((-\arctan{\alpha})x)}\right|} + (-\arctan{\alpha})t.
    \end{equation}
    Let $\Omega_b := [-1,1]\times[-b,b]$ for some large $b>0$. Namely, $\Omega_b$ is supposed to be a long cylinder. 
    Set $E_t:=\{(x,y)\in\Omega_b\ \mid\ y\leq u_\alpha(x,t)\}$ for each $t\geq 0$.
    Then, it holds that $\chambolleSet{h}(E_t) = E_{t+h}$ for every $t\geq 0$ whenever $-(t+h)/\arctan{\alpha}\geq -b$
    and $\alpha = -\beta / \sqrt{1-\beta^2}$. In other words, the capillary Chambolle type scheme yields
    the translative soliton.
\end{lem}
\begin{proof}
    Let $\normalConv$ be the unit normal vector field of $\partial E_t$
    and suppose that $\normalConv$ is extended to whole $\Omega$ by $\normalConv(x,y) := \normalConv(x,u_\alpha(x,t))$ for all $y\in[-b,b]$.
    Then, we define $w := \geodis{E_t} - h\operatorname{div}\normalConv =  \geodis{E_t} - h\kappa$.
    We deduce from the assumptions that $\normalConv$ satisfies the conditions 
    on $z\in L^\infty(\Omega;\mathbb{R}^2)$ in \cite[Theorem 2]{EtoGiga2023}.
    Therefore, $w$ is a unique minimizer of $\chambolleEnergyFour{u}$ with the data $E_t$.
    We easily observe that the function $u_\alpha$ defined by $\eqref{eq:4_TS}$ 
    is an exact solution to $\eqref{eq:4_Target}$ with $\theta\equiv\arccos{\beta}$. 
    This implies that evolving $E_t$ 
    in the normal direction by its curvature 
    is equivalent to translate it downward (parallel to the $y$-axis) 
    at the speed $\arctan{\alpha}$. Hence, the resulting $\chambolleSet{h}(E_t)$ 
    is nothing but $E_{t+h}$.
\end{proof}

{
    Let us prove a key result to show the consistency of the scheme $\chambolleFunction{h}$:
}

\begin{prop}\label{prop:4_existence_super_sub_solution}
    Let $\varphi\in C^2(\closure{\Omega})$ and $h>0$. 
    Assume that there exist constants $\betahigh > 0$ and $\betalow < 0$ such that $-1 < \betalow\leq\beta\leq\betahigh <1$.
    For each $\mu\in\mathbb{R}$, 
    we set $\sublevel{\varphi}{\mu} := \{x\in{\closure{\Omega}}\mid \varphi(x)\geq\mu\}$.
    Assume that $\grad{\varphi}(z)\neq\zerovec$ for some $z\in\closure{\Omega}$.
    If either $z\in\Omega$ or $z\in\pOmega$ and 
    $\innerproductSingle{\grad{\varphi}(z)}{\normalConv_\Omega(z)} + \beta{(z)}|\grad{\varphi}(z)| > 0$ (resp, $\innerproductSingle{\grad{\varphi}(z)}{\normalConv_{\Omega}(z)} + \beta{(z)}|\grad{\varphi}(z)| < 0$),
    then, up to a modification of $\varphi$ in a neighborhood of $z$, 
    the problem \eqref{eq:4_Neumann_problem_master} with $g := \geodis{\sublevel{\varphi}{\varphi(z)}}$ has a viscosity supersolution $\supersolution{w}$ 
    (resp, subsolution $\subsolution{w}$)
    satisfying the following condition:
    \begin{itemize}
        \item For every $\varepsilon > 0$, there exist $\delta > 0$, $h_0 > 0$, $r>0$ and $C > 0$ such that
        \begin{align}
            &\supersolution{w} \leq \geodis{\sublevel{\varphi}{\lambda}} - h\kappa_{\sublevel{\varphi}{\lambda}} + 3\varepsilon h\ \ \mbox{in}\ \ U_{\delta,r}\label{eq:4_approx_curvature_super}\\
            &(\mbox{resp,}\ \ \subsolution{w} \geq \geodis{\sublevel{\varphi}{\lambda}} - h\kappa_{\sublevel{\varphi}{\lambda}} - 3\varepsilon h\ \ \mbox{in}\ \ U_{\delta,r})\label{eq:4_approx_curvature_sub}
        \end{align}
        for any $h\in(0,h_0)$ and for any $\lambda\in\mathbb{R}$ with $|\varphi(z) - \lambda|\leq C\sqrt{h}$, where
        \begin{equation*}
            U_{\delta,r} := \{x\in\Omega\mid x\in B(z,\delta)\ \mbox{and}\ |\geodis{\sublevel{\varphi}{\lambda}}(x)|< r\}.
        \end{equation*}
        \item In particular, it holds that
        \begin{equation*}
            \left|\weaksol{h}{\sublevel{\varphi}{\lambda}} - \geodis{\sublevel{\varphi}{\lambda}} + h\kappa_{\sublevel{\varphi}{\lambda}}\right|\leq \varepsilon h\ \ \mbox{in}\ \ U_{\delta,r}.
        \end{equation*}
    \end{itemize}
\end{prop}
\begin{proof}
    The proof is a modification of \cite[Proposition 5.2]{EtoGigaIshii2012}.
    First, we treat the case where $z\in\Omega$. 
    We define $\soliton{\mu,\beta}(x) := \solitonprofile{\beta}{x'} - x_d + \mu$ for each $x\in\closure{\Omega}$.
    Since $\grad{\varphi}(z) \neq \zerovec$,
    there exists a $\delta > 0$ for which $\{\varphi = \mu\}\cap B(z,3\delta)$ is a smooth hypersurface.
    We introduce a smooth cutoff function $\eta:\Omega\to [0,1]$ satisfying
    \begin{equation*}
        \eta(x) = \begin{cases}
            1\ \ \mbox{if}\ \ x\in B(z,\delta),\\
            0\ \ \mbox{if}\ \ x\in \Omega\backslash\closure{B(z,2\delta)}.
        \end{cases}
    \end{equation*}
    Then, we replace $\varphi$ with $(1-\eta)\soliton{\mu,\beta} + \eta\varphi$. We still write it as $\varphi$ for simplicity.
    We observe that $\soliton{\mu+h,\betalow}$ is a classical supersolution to \eqref{eq:4_Neumann_problem_master} 
    with $g:=\geodis{\sublevel{\varphi}{\mu}}$.
    Indeed, we have
    \begin{equation*}
        \soliton{\mu+h,\betalow} - h\operatorname{div}{\grad\anisotropy{}{}{\grad{\soliton{\mu+h,\betalow}}}} = \geodis{\sublevel{\soliton{\mu+h,\betalow}}{\mu}} \geq \geodis{\sublevel{\varphi}{\mu}}\ \ \mbox{in}\ \ \Omega.
    \end{equation*}
    Here, we note that $\sublevel{\soliton{\mu+h,\betalow}}{\mu}\subset\sublevel{\varphi}{\mu}$ hence $\geodis{\sublevel{\soliton{\mu+h,\betalow}}{\mu}}\geq \geodis{\sublevel{\varphi}{\mu}}$.
    Moreover, we derive
    \begin{equation*}
        \innerproductSingle{\grad{\soliton{\mu+h,\betalow}}}{\normalConv_\Omega} + \beta|\grad{\soliton{\mu+h,\betalow}}| \geq \innerproductSingle{\grad{\soliton{\mu+h,\betalow}}}{\normalConv_\Omega} + \betalow|\grad{\soliton{\mu+h,\betalow}}| = 0\ \ \mbox{on}\ \ \pOmega.
    \end{equation*}
    We shall construct a viscosity supersolution to $\eqref{eq:4_Neumann_problem_master}$ in a neighborhood of $z$.
    To this end, we introduce a smooth cutoff function $\tilde{\eta}:\Omega\to[0,1]$ satisfying:
    \begin{equation*}
        \begin{array}{cc}
            \tilde{\eta}(x) = \begin{cases}
                1\ \ \mbox{if}\ \ |\geodis{\sublevel{\varphi}{\mu}}(x)|\leq r,\\
                0\ \ \mbox{if}\ \ |\geodis{\sublevel{\varphi}{\mu}}(x)|\geq 2r,
            \end{cases} &
            \|\grad{\tilde{\eta}}\|_\infty + \|\hess{\tilde{\eta}}\|_\infty\leq L,
        \end{array}
    \end{equation*}
    where $L>0$ is independent of $\varepsilon,h,$ and $r$.
    Moreover, suppose that $\tau(\varepsilon)\downarrow 0$ as $\varepsilon\downarrow 0$.
    Then, we define
    \begin{equation*}
        \tilde{w} := \geodis{\sublevel{\varphi}{\mu}} - h\tilde{\eta}\kappa_{\tau(\varepsilon)} + (1-\tilde{\eta})h + 2\varepsilon h.
    \end{equation*}
    Here, we have used the notation that $\kappa_\tau := \rho_\tau * \kappa_{\sublevel{\varphi}{\mu}}$ with the standard mollifying kernel $\rho_\tau$.
    Take $\tau(\varepsilon)$ so small that $\|\kappa_{\tau(\varepsilon)} - \kappa_{\sublevel{\varphi}{\mu}}\|_{C(\sublevel{\varphi}{\mu})} < \varepsilon$.
    Then, as discussed in \cite[Proposition 5.2]{EtoGigaIshii2012}, the function $\tilde{w}$ turns out to be
    a classical supersolution of $\eqref{eq:4_Neumann_problem_master}$ with $g:=\geodis{\sublevel{\varphi}{\mu}}$ in $U_{\delta,r}$.
    In terms of $\soliton{\mu+h,\betalow}$ and $\tilde{w}$, we define
    \begin{equation*}
        \supersolution{w} := \begin{cases}
            \min\{\tilde{w}, \soliton{\mu+h,\betalow}\}\ \ \mbox{in}\ \ \closure{U_{\delta,r}},\\
            \soliton{\mu+h,\betalow}\ \ \mbox{in}\ \ \closure{\Omega}\backslash U_{\delta,r}.
        \end{cases}
    \end{equation*}
    Since viscosity supersolutions are closed under taking minimum, we see that
    $\supersolution{w}$ is also a viscosity supersolution of $\eqref{eq:4_Neumann_problem_master}$ with $g:=\geodis{\sublevel{\varphi}{\mu}}$.
    Thus, we deduce that
    \begin{equation*}
        \supersolution{w} \leq \tilde{w} = \geodis{\sublevel{\varphi}{\mu}} - h\kappa_{\tau(\varepsilon)} + 2\varepsilon h \leq \geodis{\sublevel{\varphi}{\mu}} - h\kappa_{\sublevel{\varphi}{\mu}} + 3\varepsilon h\ \ \mbox{in}\ \ U_{\delta,r}.
    \end{equation*}
    Here, we should take $\delta > 0$ so small that $B(z,\delta)\subset\{|d_{\sublevel{\varphi}{\mu}}| < r\}$ if necessary.
    Consequently, we derive the desired $\supersolution{w}$. The comparison principle for viscosity solution{s}
    implies $\weaksol{h}{\sublevel{\varphi}{\mu}}\leq \supersolution{w}$ and hence
    \begin{equation*}
        \weaksol{h}{\sublevel{\varphi}{\mu}} \leq \tilde{w} = \geodis{\sublevel{\varphi}{\mu}} - h\kappa_{\tau(\varepsilon)} + 2\varepsilon h \leq \geodis{\sublevel{\varphi}{\mu}} - h\kappa_{\sublevel{\varphi}{\mu}} + 3\varepsilon h\ \ \mbox{in}\ \ U_{\delta,r}.
    \end{equation*}
    If $z\in\pOmega$, $\varphi$ is already a supersolution of the second equality of $\eqref{eq:4_Neumann_problem_master}$.
    Moreover, we see that the hypersurface $\{\varphi = \mu\}\cap B(z,3\delta)$ intersects 
    $\pOmega$ with the angle larger than $\arccos{\beta}$.
    Thus, we can find a translative soliton whose level set is included in $\sublevel{\varphi}{\mu}$ 
    in a neighborhood of $z$. Hence, the whole argument for $z\in\Omega$ will work. A desired viscosity subsolution 
    can be obtained in the same manner. We conclude the proof.
\end{proof}

\begin{rem}
    Let us mention difference from the researches \cite{EtoGigaIshii2012,EtoGigaIshii2012_2} 
    regarding construction of a sub- and supersolution which approximate a solution of 
    $\eqref{eq:4_Neumann_problem_master}$ in the case where $\Omega = \mathbb{R}^d$.
    Note that the boundary condition in \eqref{eq:4_Neumann_problem_master} vanishes due to $\pOmega$ is an empty set.
    Therein, the hypersurface $\{\varphi = \mu\}$ was approximated with the help of an open bounded set $V_1\subset\mathbb{R}^d$
    whose boundary is tangent to $\{\varphi = \mu\}$. Then, the function $\geodis{V_1}$ was bounded
    by rigorous solutions of $\eqref{eq:4_Neumann_problem_master}$ with $\Omega := \mathbb{R}^d$, $g := \geodis{B}$ and a ball $B$. 
    This solution was explicitly computed in \cite[\S B]{CasellesChambolle2006}.
    However, this result is not available in our case due to the boundary condition.
    We are forced to modify a test function $\varphi$ to cope the boundary condition. But, we should notice that
    the definition of viscosity solution{s} only uses local information of $\varphi$. 
    Thus, this modification does not affect the following discussion.
\end{rem}



The following lemma establishes a kind of monotonicity of the scheme $\chambolleFunction{h}$ with respect
to the contact angle function $\beta$:

\begin{lem}\label{lem:4_monotonicity_sh_beta}
    Suppose that $\beta_1:\pOmega\to[-1,1]$ and $\beta_2:\pOmega\to[-1,1]$ satisfy $\beta_1\leq \beta_2$ on $\pOmega$.
    Let $\chambolleFunction{h,b}$ be the associated function operator which is induced from the solution
    to \eqref{eq:4_Neumann_problem_master} with $\beta := b$ and $g = \geodis{E}$ for each function $b:\pOmega\to[-1,1]$.
    Then, it holds that 
    \begin{equation}\label{eq:4_monotonicity_sh_beta_1}
        \chambolleFunction{h,\beta_2}\varphi\leq \chambolleFunction{h,\beta_1}\varphi\qquad\mbox{in}\qquad\closure{\Omega}
    \end{equation}
    for any function $\varphi\in C(\closure{\Omega})$.
\end{lem}
\begin{proof}
    For each $b:\pOmega\to[-1,1]$, let $\chambolleSet{h,b}(E) := \{\weaksol{h}{E,b}\leq 0\}$ where
    $\weaksol{h}{E,b}$ the unique solution to \eqref{eq:4_Neumann_problem_master} with $\beta := b$ and $g = \geodis{E}$.
    Then, we observe that $\weaksol{h}{E,\beta_1}$ is a viscosity subsolution of \eqref{eq:4_Neumann_problem_master}
    with $\beta := \beta_2$. Hence, we deduce from the comparison principle that $\weaksol{h}{E,\beta_1}\leq \weaksol{h}{E,\beta_2}$ in $\closure{\Omega}$.
    This estimate implies that $\chambolleSet{h,\beta_2}(E)\subset\chambolleSet{h,\beta_1}(E)$. Therefore,
    by the definition of $\chambolleFunction{h}$, it follows that $\chambolleFunction{h,\beta_2}\varphi\leq\chambolleFunction{h,\beta_1}\varphi$.
\end{proof}

We are now in the position to prove the consistency of the scheme $\chambolleFunction{h}$.

\begin{thm}[Consistency of $\chambolleFunction{h}$]\label{thm:4_generator_2}
    Let $\varphi\in C^2(\closure{\Omega})$.
    Suppose that $\Omega$ and $\beta$ satisfy the criteria of Proposition \ref{prop:4_existence_super_sub_solution} and \Theorem{thm:4_grad_bound_of_weaksol}.
    Then, it holds that
    \begin{align}
        &\relaxsup{h}{0}\frac{\chambolleFunction{h}\varphi(z) - \varphi(z)}{h}\leq -F_*(\grad{\varphi}(z),\hess{\varphi}(z))\label{eq:4_limsup_estimate}\\
        &(resp,\quad\relaxinf{h}{0}\frac{\chambolleFunction{h}\varphi(z) - \varphi(z)}{h} \geq -F^*(\grad{\varphi}(z),\hess{\varphi}(z)))\label{eq:4_liminf_estimate}
    \end{align}
    whenever $z\in\closure{\Omega}$ satisfies one of the following conditions:
    \begin{itemize}
        \item $z\in\Omega$ and either $\nabla\varphi(z)\neq 0$ or $\nabla\varphi(z) = 0$ and $\nabla^2\varphi(z) = O$.
        \item $z\in\pOmega$, $\nabla\varphi(z)\neq 0$ and $\innerproductSingle{\grad{\varphi(z)}}{\normalConv_\Omega(z)} + \beta(z)|\nabla\varphi(z)| > 0$ \newline
        (resp, $\innerproductSingle{\grad{\varphi(z)}}{\normalConv_\Omega(z)} + \beta(z)|\grad{\varphi(z)}| < 0$).
    \end{itemize}
\end{thm}
\begin{proof}
    Set $\mu := \varphi(z)$. Fix any $\varphi\in C^2(\closure{\Omega})$ and any $\varepsilon > 0$.\newline\newline
    [\textbf{Case $z\in\Omega$ and $\grad{\varphi}(z)\neq \zerovec$}]\newline
    Then, we deduce from Lemma \ref{prop:4_existence_super_sub_solution} that
    there exist a smooth function $\tilde{\varphi}$ and a positive constant $\delta$ such that the estimate
    \begin{equation}\label{eq:4_approx_estimate}
        |\weaksol{h}{\sublevel{\tilde{\varphi}}{\lambda}} - \geodis{\sublevel{\tilde{\varphi}}{\lambda}} + h\kappa_{\sublevel{\tilde{\varphi}}{\lambda}}| \leq \varepsilon h\ \ \mbox{in}\ \ U_{\delta,r}
    \end{equation}
    holds for sufficiently small $h>0$ and $r > 0$ and $\lambda\in\mathbb{R}$ with $|\mu-\lambda|\leq C\sqrt{h}$ with $C:=\sqrt{2}|\grad{\varphi(z)}|$
    and $\tilde{\varphi} = \varphi$ in $U_{\delta,r}$. 
    For simplicity, we still write $\tilde{\varphi}$ as $\varphi$.
    We now define
    \begin{equation*}
        \lambda_h^\pm := \varphi(z_h^\pm) +\{-F(\grad{\varphi(z)},\hess{\varphi(z)}) + \varepsilon\} h,
    \end{equation*}
    where
    \begin{equation*}
        z_h^\pm := z \pm \frac{\grad{\varphi}(z)}{|\grad{\varphi}(z)|}\sqrt{2h}.
    \end{equation*}
    Then, we shall show that $\chambolleFunction{h}\varphi(z_h^\pm) \leq\lambda_h^\pm$ for sufficiently small $h>0$.
    We note that this statement is equivalent to $z_h^\pm\notin\chambolleSet{h}(\sublevel{\varphi}{\lambda_h^\pm})$ by Lemma \ref{lem:4_Sh_Th}.
    First, we prove that $\chambolleFunction{h}\varphi(z_h^-) \leq \lambda_h^-$.
    We use \eqref{eq:4_approx_estimate} with $\mu:=\lambda_h^-$ to derive
    \begin{equation}\label{eq:4_approx_estimate_2}
        \weaksol{h}{\sublevel{\varphi}{\lambda_h^-}}(z_h^-)\geq \geodis{\sublevel{\varphi}{\lambda_h^-}}(z_h^-) - h\kappa_{\sublevel{\varphi}{\lambda_h^-}}(z_h^-) - \varepsilon h
        \geq \geodis{\sublevel{\varphi}{\lambda_h^-}}(z_h^-) - (K+\varepsilon)h,
    \end{equation}
    where $K := \sup_{0\leq h\leq 1}\|\kappa_{\sublevel{\varphi}{\lambda_h^-}}\|_{C(\overline{U_{\delta,r}})}$.
    Since $\geodis{\sublevel{\varphi}{\lambda_h^-}}$ is smooth, we have
    \begin{equation}\label{eq:4_approx_estimate_3}
        \geodis{\sublevel{\varphi}{\lambda_h^-}}(z_h^-) 
        = \geodis{\sublevel{\varphi}{\lambda_h^-}}(z) 
        -\innerproductSingle{\grad{\geodis{\sublevel{\varphi}{\lambda_h^-}}}(\widetilde{z_h^-})}{\frac{\grad{\varphi}(z)}{|\grad{\varphi}(z)|}}\sqrt{2h},
    \end{equation}
    where $\widetilde{z_h^-} = z - \frac{\grad{\varphi(z)}}{|\grad{\varphi(z)}|}\widetilde{h}$ for some $\widetilde{h}\in (0,\sqrt{2h})$.
    We deduce from the geometry (see \Figure{fig:4_interior}) that
    \begin{equation}\label{eq:4_approx_estimate_4}
        \geodis{\sublevel{\varphi}{\lambda_h^-}}(z)\to 0\qquad\mbox{and}\qquad \innerproductSingle{\grad{\geodis{\sublevel{\varphi}{\lambda_h^-}}}(\widetilde{z_h^-})}{\frac{\grad{\varphi}(z)}{|\grad{\varphi}(z)|}}\to -1
    \end{equation}
    as $h\to 0$. Here, we have recalled that $|\grad{\geodis{\sublevel{\varphi}{\lambda_h^-}}}| = 1$ 
    to derive the second convergence of \eqref{eq:4_approx_estimate_4}.

  \begin{figure}[H]
      \centering
      \includegraphics[keepaspectratio, scale=0.3]{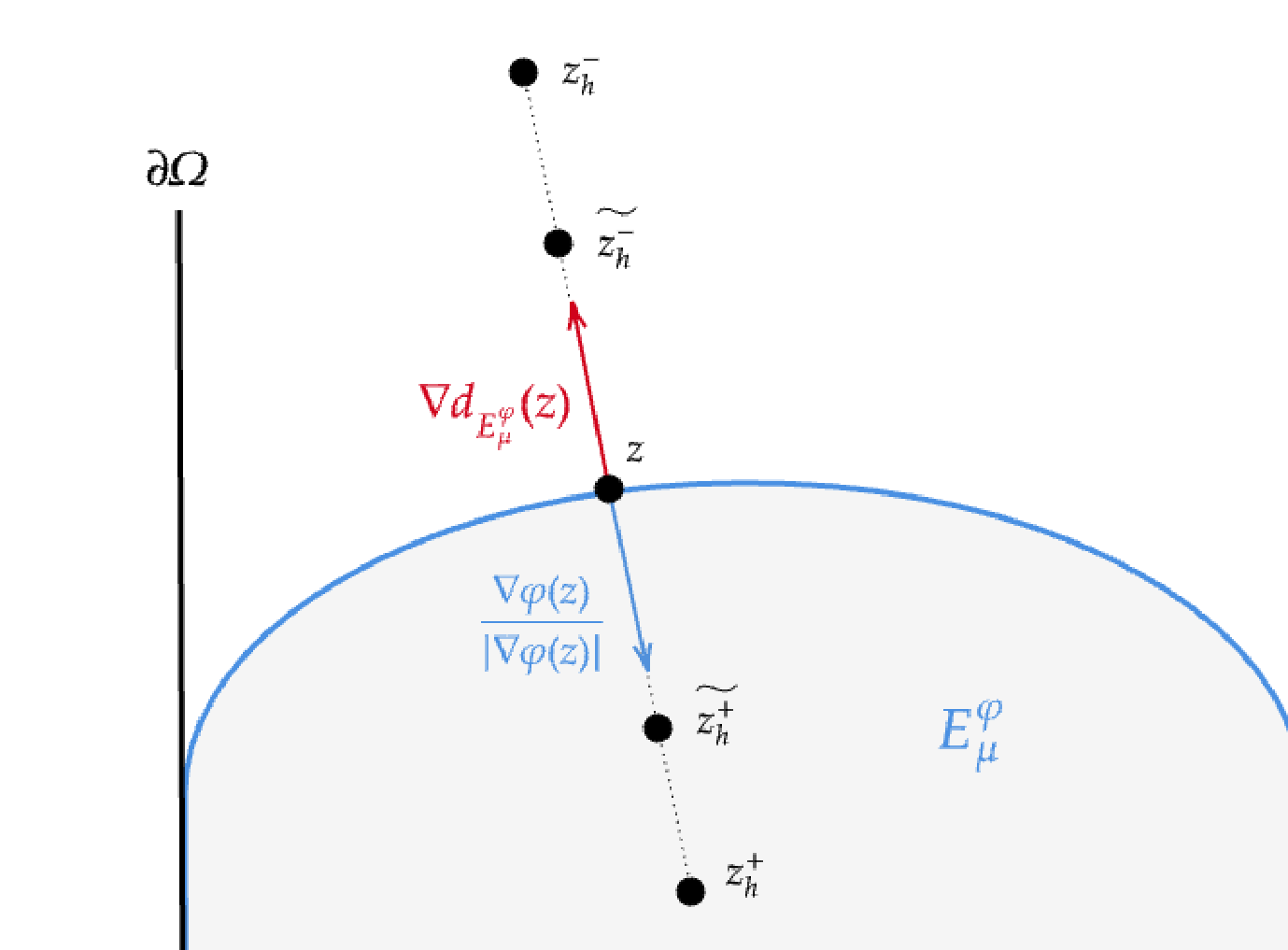}
      \caption{The location of important points associated with $z\in\Omega$.}\label{fig:4_interior}
  \end{figure}

    Combining \eqref{eq:4_approx_estimate_2}, \eqref{eq:4_approx_estimate_3} and \eqref{eq:4_approx_estimate_4},
    we conclude that $\weaksol{h}{\sublevel{\varphi}{\lambda_h^-}}(z_h^-) > 0$ for sufficiently small $h>0$.
    Thus, we obtain that $z_h^-\notin\chambolleSet{h}(\sublevel{\varphi}{\lambda^-_h})$.

    Second, we show that $\chambolleFunction{h}\varphi(z_h^+)\leq\lambda_h^+$.
    {Comparing the super level sets $\sublevel{\varphi}{\mu}$ and $\sublevel{\soliton{\mu+\frac{\clow h}{2},\betalow}}{\mu}$ (see \Figure{fig:4_compare_graphs}) and applying} Lemma \ref{lem:4_monotonicity_sh_beta}, {we compute}
    \begin{equation}\label{eq:4_approx_estimate_5}
        \chambolleFunction{h}\varphi \leq \chambolleFunction{h}\left(\soliton{\mu+\frac{\clow h}{2},\betalow}\right) \leq \chambolleFunction{h,\betalow}\left(\soliton{\mu+\frac{\clow h}{2},\betalow}\right) = \soliton{\mu-\frac{\clow h}{2},\betalow}\qquad\mbox{in}\qquad\closure{\Omega},
    \end{equation}
    where 
    \begin{equation*}
        \clow := \arctan{\left(\frac{-\betalow}{\sqrt{1 - \betalow^2}}\right)}.
    \end{equation*}

  \begin{figure}[H]
      \centering
      \includegraphics[keepaspectratio, scale=0.15]{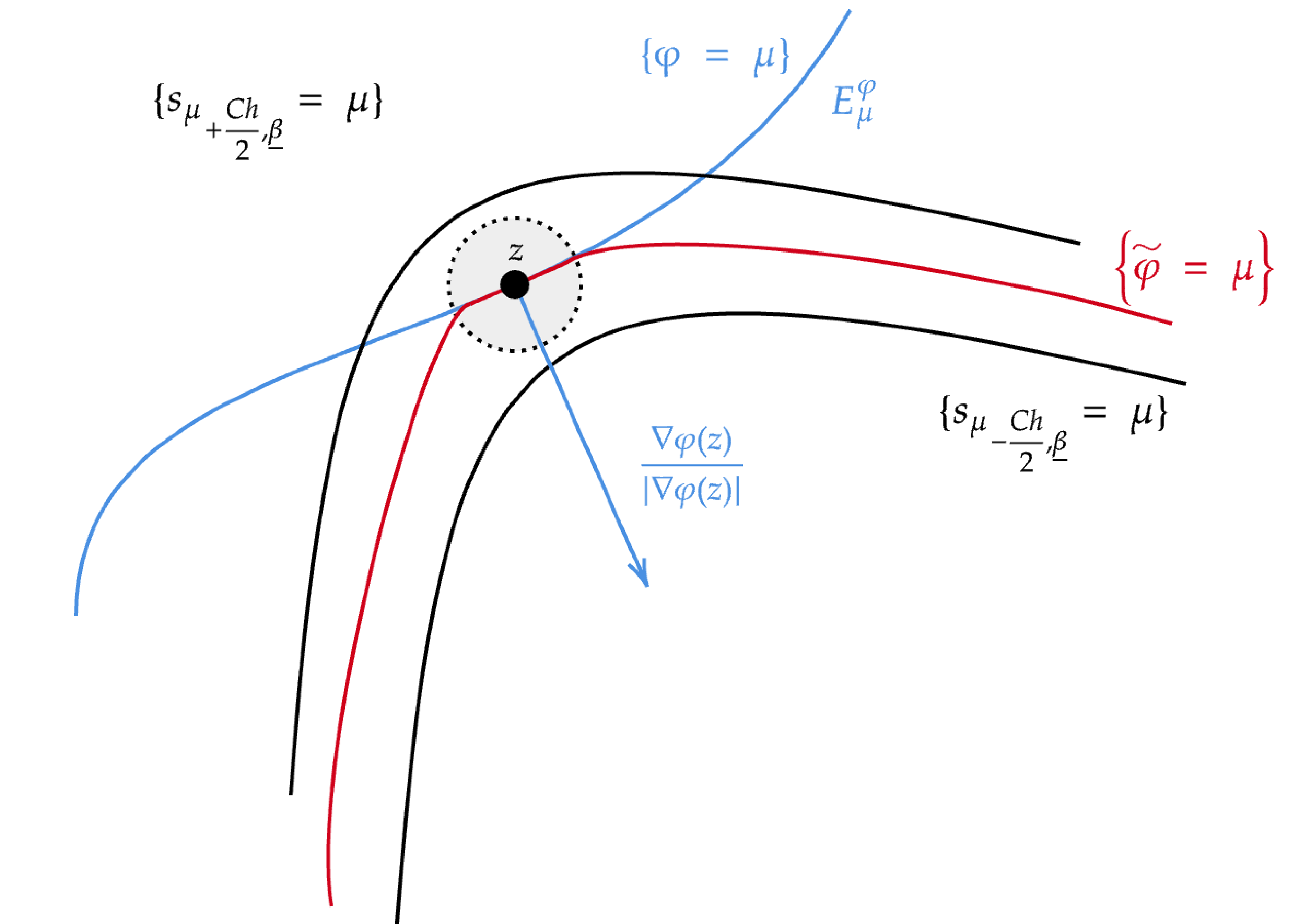}
      \caption{{The boundaries of the super level sets.}}\label{fig:4_compare_graphs}
  \end{figure}

    Evaluating \eqref{eq:4_approx_estimate_5} at $z_h^+$ yields
    \begin{equation}\label{eq:4_approx_estimate_5_1}
        \chambolleFunction{h}\varphi(z_h^+) \leq \soliton{\mu - \frac{\clow h}{2},\betalow}(z_h^+) = \mu + |\grad{\varphi(z)}|\frac{\clow h}{2}\leq \mu + |\grad{\varphi(z)}|\sqrt{2h} + h\laplacian{\varphi(z)}
    \end{equation}
    for sufficiently small $h>0$.
    Here, to derive the last inequality of \eqref{eq:4_approx_estimate_5_1}, we note that
    for every $C_1\in\mathbb{R}$ and $C_2 > 0$, $C_1h < C_2\sqrt{h}$ holds for sufficiently small $h>0$
    (we may set $C_1 := |\grad{\varphi(z)}|\frac{\clow}{2} - \laplacian{\varphi(z)}$ and $C_2 := \sqrt{2}|\grad{\varphi}{(z)}|$).
    Meanwhile, {the Taylor} expansion gives
    \begin{equation}\label{eq:4_approx_estimate_5_2}
        \varphi(z_h^+) = \mu + |\grad{\varphi(z)}|\sqrt{2h} + \innerproductSingle{\hess{\varphi(\widetilde{z_h^+})\frac{\grad{\varphi(z)}}{|\grad{\varphi(z)}|}}}{\frac{\grad{\varphi(z)}}{|\grad{\varphi(z)}|}}h,
    \end{equation}
    where $\widetilde{z_h^+} = z + \frac{\grad{\varphi(z)}}{|\grad{\varphi(z)}|}\widetilde{h}$ for some $\widetilde{h}\in (0,\sqrt{2h})$.
    Combining \eqref{eq:4_approx_estimate_5_1} and \eqref{eq:4_approx_estimate_5_2}, we derive
    \begin{equation*}
        \chambolleFunction{h}\varphi(z_h^+)\leq \varphi(z_h^+) + \left\{\laplacian{\varphi(z)} - \innerproductSingle{\hess{\varphi({\widetilde{z_h^+}})}\frac{\grad{\varphi(z)}}{|\grad{\varphi(z)}|}}{\frac{\grad{\varphi(z)}}{|\grad{\varphi(z)}|}}\right\}h.
    \end{equation*}
    Noting that
    \begin{equation*}
        -F(\grad{\varphi(z)},\hess{\varphi(z)}) = \laplacian{\varphi(z)} - \innerproductSingle{\hess{\varphi(z)}\frac{\grad{\varphi(z)}}{|\grad{\varphi(z)}|}}{\frac{\grad{\varphi(z)}}{|\grad{\varphi(z)}|}}
    \end{equation*}
    and that $\hess{\varphi}$ is continuous, 
    we can take $h>0$ so small that
    \begin{equation*}
        \chambolleFunction{h}\varphi(z_h^+) \leq \varphi(z_h^+) + \{-F(\grad{\varphi(z),\hess{\varphi(z)}}) + \varepsilon\}h.
    \end{equation*}
    The similar argument yields
    \begin{equation*}
        \chambolleFunction{h}\varphi(z_h^\pm) \geq \varphi(z_h^\pm) + \{-F(\grad{\varphi(z)}, \hess{\varphi(z)}) - \varepsilon\}h
    \end{equation*}
    for sufficiently small $h>0$. We complete the proof for this case.
    \newline\newline
    \noindent
    [\textbf{Case $z\in\pOmega$ and $\grad{\varphi(z)}\neq \zerovec$}]\newline
    In this case, we have to consider a sequence $\{z_h\}_h$ on $\pOmega$ which converges to $z$.
    To this end, for each $z\in\pOmega$, we set 
    \begin{equation*}
        \containerTangent(z) := \frac{\grad{\varphi(z)}}{|\grad{\varphi(z)}|} - \Phi(z)\normalConv_\Omega(z)\qquad\mbox{with}\qquad\Phi(z):=\innerproductSingle{\frac{\grad{\varphi(z)}}{|\grad{\varphi(z)}|}}{\normalConv_\Omega(z)}.
    \end{equation*}
    Note that $\containerTangent(z)$ is nothing but the projection of the vector $\normalConv(z)$ onto $\pOmega$.
    Then, we define
    \begin{equation*}
        z_h^\pm := z \pm \containerTangent(z)\sqrt{2h}.
    \end{equation*}
    First, suppose that $\innerproductSingle{\grad{\varphi(z)}}{\normalConv_\Omega(z)} + \beta(z)|\grad{\varphi(z)}| > 0$.
    Then, we have
    \begin{equation*}
        \Phi(z) > -\beta(z)\geq -\betahigh.
    \end{equation*}
    Then, geometrically speaking, it holds that either the graph of $\varphi$ is bounded by $\soliton{\mu,-\betahigh}$ from above or that
    it is bounded from $\soliton{\mu,\betahigh}$ below. See \Figure{fig:4_soliton_image_1} to grasp the situation.

  \begin{figure}[H]
      \centering
      \includegraphics[keepaspectratio, scale=0.15]{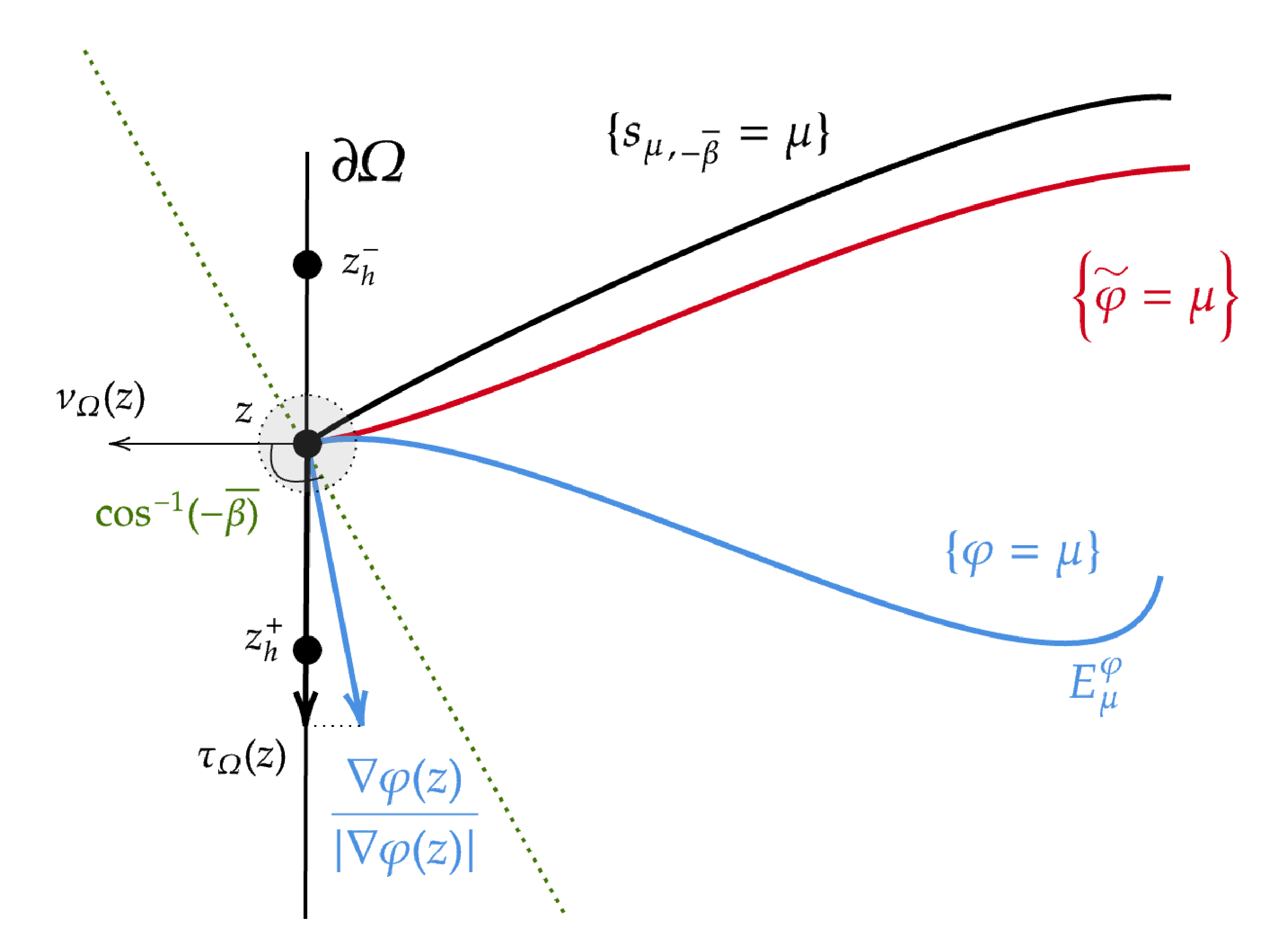}
      \caption{The graphs of a translating soliton and the level set of $\varphi$.}\label{fig:4_soliton_image_1}
  \end{figure}

    We assume that the former statement is valid.
    Then, comparing the level sets of $\varphi$ and $\soliton{\mu,-\betahigh}$, we deduce that
    \begin{equation}\label{eq:4_approx_estimate_6}
        \soliton{\mu,-\betahigh} \geq \varphi\qquad\mbox{in}\qquad\closure{\Omega}.
    \end{equation}
    Applying $\chambolleFunction{h}$ both sides of \eqref{eq:4_approx_estimate_6} 
    and the monotonicity of $\chambolleFunction{h}$ together with Lemma \ref{lem:4_monotonicity_sh_beta} yield
    \begin{equation}\label{eq:4_approx_estimate_7}
        \soliton{\mu-\chigh h,-\betahigh} = \chambolleFunction{h,-\betahigh}(\soliton{\mu,-\betahigh}) \geq \chambolleFunction{h}(\soliton{\mu,-\betahigh})\geq \chambolleFunction{h}\varphi,
    \end{equation}
    where
    \begin{equation*}
        \chigh := \arctan{\left(\frac{\betahigh}{\sqrt{1-\betahigh^2}}\right)}.
    \end{equation*}
    We evaluate the equation \eqref{eq:4_approx_estimate_7} at $z_h^+$ to obtain
    \begin{align}\label{eq:4_approx_estimate_13}
        \begin{split}
            \chambolleFunction{h}\varphi(z_h^+) &\leq \soliton{\mu-\chigh h-\betahigh}(z_h^+) = \varphi(z) + \chigh h\innerproductSingle{\containerTangent(z)}{\grad{\varphi(z)}}\\
            &\leq \mu + \innerproductSingle{\containerTangent(z)}{\grad{\varphi(z)}}\sqrt{2h} + h\laplacian{\varphi(z)}\\
            & \quad +2h\Phi(z)\innerproductSingle{\hess{\varphi(z)}\frac{\grad{\varphi(z)}}{|\grad{\varphi(z)}|}}{\normalConv_\Omega(z)} - h\Phi(z)^2\innerproductSingle{\hess{\varphi(z)}\normalConv_\Omega(z)}{\normalConv_\Omega(z)},
        \end{split}
    \end{align}
    where $h > 0$ is taken small enough, and we have used that $\innerproductSingle{\containerTangent(z)}{\grad{\varphi(z)}} > 0$.
    Meanwhile, {the Taylor} expansion shows
    \begin{equation}\label{eq:4_approx_estimate_14}
        \varphi(z_h^+) = \mu + \innerproductSingle{\containerTangent(z)}{\grad{\varphi(z)}}\sqrt{2h} + \innerproductSingle{\hess{\varphi(\widetilde{z_h^+})}\containerTangent(z)}{\containerTangent(z)}h
    \end{equation}
    and
    \begin{align}\label{eq:4_approx_estimate_15}
        \begin{split}
            \innerproductSingle{\hess{\varphi(\widetilde{z_h^+})}\containerTangent(z)}{\containerTangent(z)} &= \innerproductSingle{\hess{\varphi(\widetilde{z_h^+})}\frac{\grad{\varphi(z)}}{|\grad{\varphi(z)}|}}{\frac{\grad{\varphi(z)}}{|\grad{\varphi(z)}|}}\\
            &\quad +2\Phi(z)\innerproductSingle{\hess{\varphi(\widetilde{z_h^+})}\frac{\grad{\varphi(z)}}{|\grad{\varphi(z)}|}}{\normalConv_\Omega(z)} \\
            &\quad -\Phi(z)^2\innerproductSingle{\hess{\varphi(\widetilde{z_h^+})}\normalConv_\Omega(z)}{\normalConv_\Omega(z)},
        \end{split}
    \end{align}
    where $\widetilde{z_h^+} = z + \frac{\grad{\varphi(z)}}{|\grad{\varphi(z)}|}\widetilde{h}$ with $\widetilde{h}\in(0,\sqrt{2h})$.
    Since $\hess{\varphi(z)}$ is continuous, we deduce from \eqref{eq:4_approx_estimate_13}, \eqref{eq:4_approx_estimate_14} and \eqref{eq:4_approx_estimate_15} that
    \begin{align*}
        \chambolleFunction{h}\varphi(z_h^+)&\leq \varphi(z_h^+) + \left\{\laplacian{\varphi(z)} -\innerproductSingle{\hess{\varphi(\widetilde{z_h^+})}\surfacenormal{\varphi}{z}}{\surfacenormal{\varphi}{z}}\right\}h\\
        &\quad +2\Phi(z)\innerproductSingle{\left\{\hess{\varphi(\widetilde{z_h^+})} - \hess{\varphi(z)}\right\}\surfacenormal{\varphi}{z}}{\normalConv_\Omega(z)}h \\
        &\quad +\Phi(z)^2\innerproductSingle{\left\{\hess{\varphi(\widetilde{z_h^+})} - \hess{\varphi(z)}\right\}\normalConv_\Omega(z)}{\normalConv_\Omega(z)}h\\
        &\leq \varphi(z_h^+) + \left\{-F(\grad{\varphi(z),\hess{\varphi(z)}}) + \varepsilon\right\}
    \end{align*}
    for sufficiently small $h > 0$.
    Let us estimate $\chambolleFunction{h}\varphi$ at $z_h^-$.
    Fix any $\varepsilon > 0$, we set $\lambda^-_h := \varphi(z_h^-) + \{-F(\grad{\varphi(z),\hess{\varphi(z)}}) + \varepsilon\} h$.
    Then, it suffices to prove that $\weaksol{h}{\sublevel{\varphi}{\lambda^-_h}}(z_h^-) > 0$ for $h > 0$ small enough.
    We deduce from Proposition \ref{prop:4_existence_super_sub_solution} that
    \begin{align}\label{eq:4_approx_estimate_8}
        \weaksol{h}{\sublevel{\varphi}{\lambda_h^-}}(z_h^-) &\geq \geodis{\sublevel{\varphi}{\lambda_h^-}}(z_h^-) - h\kappa_{\sublevel{\varphi}{\lambda^-_h}}(z_h^-) -\varepsilon h\nonumber \\
        &\geq \geodis{\sublevel{\varphi}{\lambda^-_h}}(z) - \innerproductSingle{\containerTangent(z)}{\grad{\geodis{\sublevel{\varphi}{\lambda_h^-}}(z)}}\sqrt{2h} - h(K+\varepsilon).
    \end{align}
    Here, we note that the coefficient of $\sqrt{2h}$ always positive by geometry.
    Letting $h > 0$ small enough, we see that the left-hand side of \eqref{eq:4_approx_estimate_8} is positive, which means 
    \begin{equation*}
        \chambolleFunction{h}\varphi(z_h^-) \leq \lambda^-_h = \varphi(z_h^-) + \{-F(\grad{\varphi(z)}, \hess{\varphi(z)}) + \varepsilon\} h.
    \end{equation*}
    In the case where the graph of $\varphi$ is bounded by $\soliton{\mu,\betahigh}$ from below, the previous arguments still work,
    replacing $-\betahigh$ and $-\chigh$ with $\betahigh$ and $\chigh$ respectively.
    Therefore, we conclude that the estimate \eqref{eq:4_limsup_estimate} is valid whenever $z\in\pOmega$.

    Though, the estimate \eqref{eq:4_liminf_estimate} can be deduced by the similar argument, we shall show it for completeness.
    Suppose that $\innerproductSingle{\grad{\varphi(z)}}{\normalConv_\Omega(z)} + \beta(z)|\grad{\varphi(z)}| < 0$.
    Then, we see that
    \begin{equation}
        \Phi(z) < -\beta(z) \leq -\betalow.
    \end{equation}
    As discussed before, it holds that either the graph of $\varphi$ is bounded by $\soliton{\mu,\betalow}$ from above or that
    it is bounded by $\soliton{\mu,-\betalow}$ from below. We only deal with the former case.
    For $z_h^+$, we deduce from Proposition \ref{prop:4_existence_super_sub_solution} and {the Taylor} expansion that
    \begin{align}
        \weaksol{h}{\sublevel{\varphi}{\lambda^+_h}} &\leq \geodis{\sublevel{\varphi}{\lambda^+_h}}(z_h^+) + h\kappa_{\sublevel{\varphi}{\lambda_h^+}} + \varepsilon h\nonumber\\
        & \leq \geodis{\sublevel{\varphi}{\lambda_h^+}}(z) + \innerproductSingle{\containerTangent(z)}{\grad{\geodis{\sublevel{\varphi}{\lambda^+_h}}(\widetilde{z_h^+})}}\sqrt{2h} + h(K+\varepsilon),\label{eq:4_approx_estimate_9}
    \end{align}
    where $\widetilde{z_h^+} = z + \containerTangent(z)\widetilde{h}$ for some $\widetilde{h}\in(0,\sqrt{2h})$.
    {Note} that the coefficient of $\sqrt{2h}$ in \eqref{eq:4_approx_estimate_9} is always negative for sufficiently small $h>0$.
    Thus, we see that $z_h^+ \in \chambolleSet{h}(\sublevel{\varphi}{\lambda_h^+})$. In other words, we obtain
    \begin{equation*}
        \chambolleFunction{h}\varphi(z_h^+) \geq \lambda_h^+ = \varphi(z_h^+) + \{-F(\grad{\varphi(z), \hess{\varphi(z)}}) - \varepsilon\} h.
    \end{equation*}
    We deduce from geometry that
    \begin{equation}\label{eq:4_approx_estimate_10}
        \soliton{\mu,\betalow} \leq \varphi\qquad\mbox{in}\qquad\closure{\Omega}.
    \end{equation}
    Applying $\chambolleFunction{h}$ to both sides of \eqref{eq:4_approx_estimate_10} and the monotonicity of $\chambolleFunction{h}$ show
    \begin{equation}\label{eq:4_approx_estimate_11}
        \soliton{\mu-\clow h,\betalow} = \chambolleFunction{h}(\soliton{\mu,\betalow})\leq \chambolleFunction{h}\varphi.
    \end{equation}
    Evaluating \eqref{eq:4_approx_estimate_11} at $z_h^-$, we derive
    \begin{equation}\label{eq:4_approx_estimate_12}
        \mu - \innerproductSingle{\containerTangent(z)}{\grad{\varphi(z)}}\sqrt{2h}\leq \mu - \clow h\innerproductSingle{\containerTangent(z)}{\grad{\varphi(z)}}
        = \soliton{\mu - \clow h,\betalow}(z_h^-) \leq \chambolleFunction{h}\varphi(z_h^-)
    \end{equation}
    for sufficiently small $h>0$. Here, we have used $\innerproductSingle{\containerTangent(z)}{\grad{\varphi(z)}}>0$ by geometry.
    Hence, we again apply {the Taylor} expansion to the left-hand side of \eqref{eq:4_approx_estimate_12} to obtain
    \begin{equation*}
        \chambolleFunction{h}\varphi(z_h^-)\geq \varphi(z_h^-) - \innerproductSingle{\hess{\varphi(\widetilde{z_h^-})}\containerTangent(z)}{\containerTangent(z)}h,
    \end{equation*}
    where $\widetilde{z_h^-} = z - \containerTangent(z)\widetilde{h}$ for some $\widetilde{h}\in(0,\sqrt{2h})$.
    In the same argument in the previous case, we deduce that
    \begin{equation*}
        \chambolleFunction{h}\varphi(z_h^-)\geq \varphi(z_h^-) + \{-F(\grad{\varphi(z),\hess{\varphi(z)}}) - \varepsilon\}h
    \end{equation*}
    for sufficiently small $h>0$.
    \newline\newline
    \noindent
    [\textbf{Case $\grad{\varphi(z)} = \zerovec$ and $\hess{\varphi}(z) = O$}]\newline
    In this case, we note that $F^*(\zerovec,O) = F_*(\zerovec,O) = 0$ (see e.g. \cite[Lemma 1.6.16]{G}).
    Thus, our aim is to prove that
    \begin{equation}\label{eq:4_approx_estimate_16}
        \relaxinf{h}{0}\frac{\chambolleFunction{h}\varphi(z) - \varphi(z)}{h} = \relaxsup{h}{0}\frac{\chambolleFunction{h}\varphi(z) - \varphi(z)}{h} = 0. 
    \end{equation}
    Fix any $\varepsilon > 0$ and take $h > 0$ so small that $h^2 < \varepsilon h$.
    We may assume that $\varphi$ equals a constant $\mu\in\mathbb{R}$ in $B(z,\varepsilon h)$
    by taking $h>0$ much smaller if necessary.
    For each $\bv v\in\mathbb{R}^d$ with $|\bv v| = 1$, we define
    \begin{equation*}
        z_h^\pm := z \pm h^2\bv v.
    \end{equation*}
    By Lemma \ref{lem:4_existence_ts}, we can choose a translative soliton which moves to the direction of $\bv v$.
    Then, we easily observe (see \Figure{fig:4_soliton_image_2}) that
    \begin{equation}\label{eq:4_approx_estimate_17}
        \soliton{\mu-\varepsilon h - \chigh h,\betahigh} \leq \varphi \leq \soliton{\mu+\varepsilon h + \clow h, \betalow}\qquad\mbox{in}\qquad B(z,\varepsilon h).
    \end{equation}

  \begin{figure}[H]
      \centering
      \includegraphics[keepaspectratio, scale=0.3]{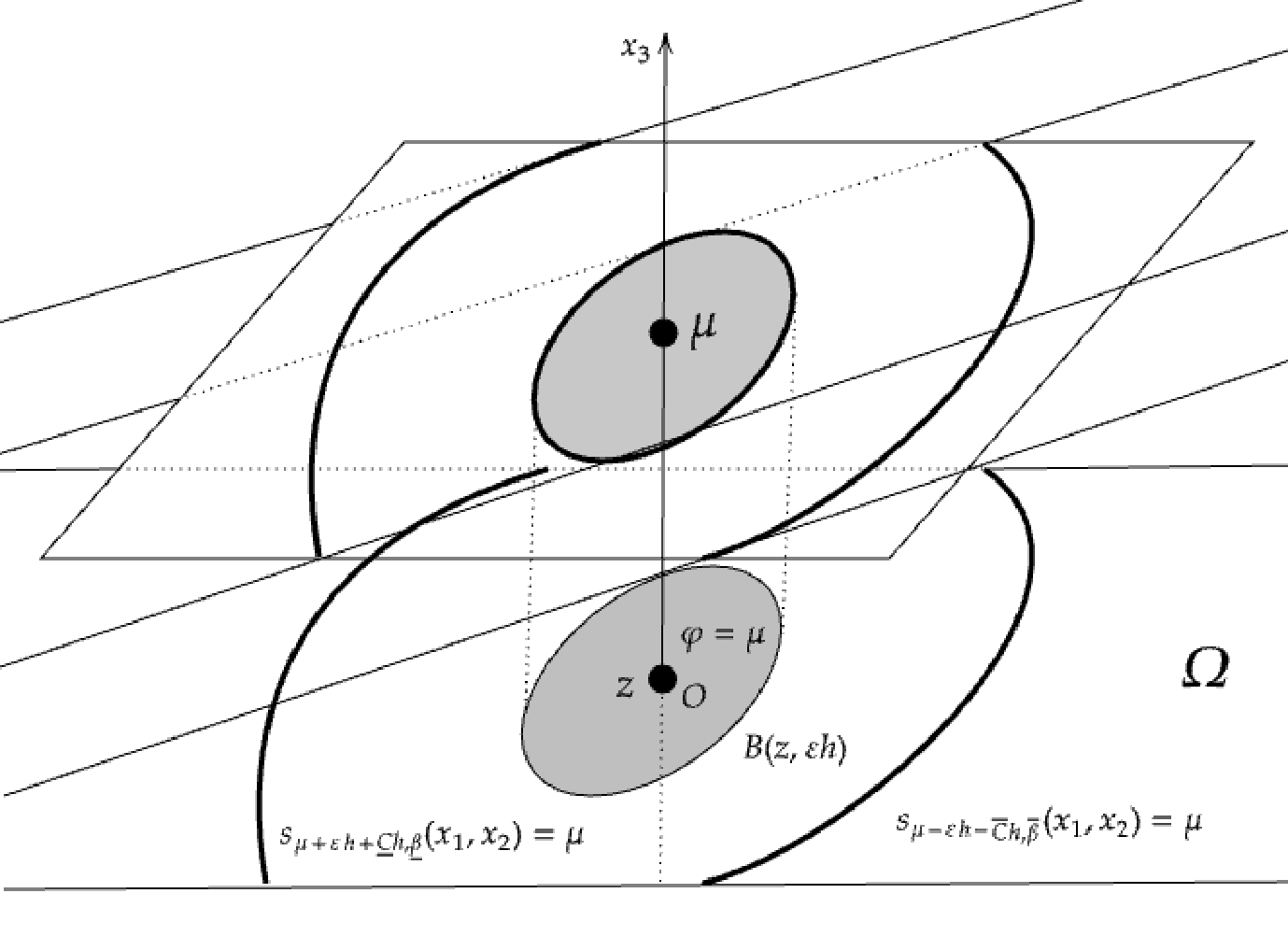}
      \caption{Level sets of $\soliton{\mu-\varepsilon h - \chigh h,\betahigh}$, $\soliton{\mu+\varepsilon h + \clow h, \betalow}$ and $\varphi$.}\label{fig:4_soliton_image_2}
  \end{figure}

    The estimate \eqref{eq:4_approx_estimate_17} holds in $\closure{\Omega}$ under modification of $\varphi$ outside $B(z,\varepsilon h)$.
    Hence, applying $\chambolleFunction{h}$ to \eqref{eq:4_approx_estimate_17} yields
    \begin{align}
        &\soliton{\mu-\varepsilon h,\betahigh}= \chambolleFunction{h,\betahigh}(\soliton{\mu-\varepsilon h - \chigh h,\betahigh})\leq \chambolleFunction{h}(\soliton{\mu-\varepsilon h - \chigh h,\betahigh}) \leq \chambolleFunction{h}\varphi,\label{eq:4_approx_estimate_18_1}\\
        &\chambolleFunction{h}\varphi \leq\chambolleFunction{h}(\soliton{\mu+\varepsilon h+\clow h,\betalow})\leq\chambolleFunction{h,\betalow}(\soliton{\mu + \varepsilon h +\clow h,\betalow}) = \soliton{\mu + \varepsilon h,\betalow}.\label{eq:4_approx_estimate_18_2}
    \end{align}
    We now evaluate \eqref{eq:4_approx_estimate_18_1} and \eqref{eq:4_approx_estimate_18_2} at $z_h^+$ to get
    \begin{align*}
        &\varphi(z_h^+) - \varepsilon h = \mu - \varepsilon h\leq \soliton{\mu - \varepsilon h,\betahigh}(z_h^+)\leq \chambolleFunction{h}\varphi(z_h^+),\\
        &\chambolleFunction{h}\varphi(z_h^+) \leq \soliton{\mu+\varepsilon h,\betalow}(z_h^+)\leq \mu + \varepsilon h = \varphi(z_h^+) + \varepsilon h.
    \end{align*}
    Here, we note that $\partial_{x_d}\soliton{\mu,k} = -1$ for every $\mu\in\mathbb{R}$ and $k\in (-1,1)$ and that $z_h^+\in B(z,\varepsilon h)$.
    The same argument works even if $z_h^+$ is replaced with $z_h^-$. Since the vector $\bv v$ can be chosen arbitrarily,
    we obtain the equality \eqref{eq:4_approx_estimate_16}.
\end{proof}
\begin{rem}
    We note that the convexity of $\Omega$ is used only to derive the equi-continuity of $\weaksol{h}{E}$ with respect to $h>0$;
    it yields the continuity of $\chambolleSet{h}$ and hence \Lemma{lem:4_Sh_Th} follows.
    Therefore, we might not need the convexity of $\Omega$ to establish the consistency of $\chambolleFunction{h}$.
\end{rem}
\section{Conclusion}\label{sec:4_Conclusion}
In this paper, we have confirmed that the capillary Chambolle type scheme,
which was proposed in \cite{EtoGiga2023}, is convergent under several assumptions on
the domain $\Omega$ and the contact angle function $\beta$.
To this end, it is crucial to derive a generator of the function operator $\chambolleFunction{h}$
due to Barles and Souganidis \cite{BarlesSouganidis1991}.
For this, we have established the equi-continuity of the minimizers $\weaksol{h}{E}$
which leads to an important relation between $\chambolleFunction{h}$ and $\chambolleSet{h}$ and
have shown that the translative soliton can be mapped to another translative soliton by $\chambolleFunction{h}$ {and $\chambolleSet{h}$}.
In the course of acquisition of the generator, we frequently use the comparison principle
of viscosity solutions and compare the hypersurface $\curve{}{t}$ induced from a test function $\varphi$ 
with translative solitons which bound $\curve{}{t}$ from above and from below, respectively.
Finally, let us give a concluding remark. When we show that the scheme is convergent, we have assumed that
$\|\beta\|_\infty$ is less than $1$. In other words, the hypersurface $\curve{}{t}$ must not be tangent to $\pOmega$.
However, we expect that $\|\beta\|_\infty$ might be allowed to equal $1$ by approximation of equations {for $\beta$ } with $\|\beta\|_\infty < 1$.
\section{Acknowledgements}
{
The work of the second author was partly supported by the Japan Society 
for the Promotion of Science (JSPS) through the grants Kakenhi: 
No.~20K20342, No.~19H00639, No.~18H05323, and by Arithmer Inc., Daikin Industries, Ltd.\ 
and Ebara Corporation through collaborative grants.
}

\bibliographystyle{siam}
\bibliography{reference}

\begin{thebibliography}{10}

\bibitem{AlmregTaylorWang1993}
{\sc F.~Almgren, J.~E. Taylor, and L.~Wang}, {\em Curvature-{D}riven {F}lows:
  {A} {V}ariational {A}pproach}, SIAM J.\ Control Optim., 31 (1993),
  pp.~387--438.

\bibitem{Mazon}
{\sc F.~Andreu-Vaillo, V.~Caselles, and J.~Maz\'{o}n}, {\em Parabolic
  quasilinear equations minimizing linear growth functionals}, vol.~223 of
  Progress in Mathematics, Birkh\"{a}user Verlag, Basel, 2004.

\bibitem{Ba}
{\sc G.~Barles}, {\em {N}onlinear {N}eumann {B}oundary {C}onditions for
  {Q}uasilinear {D}egenerate {E}lliptic {E}quations and {A}pplications}, J.\
  Differential Equations, 154 (1999), pp.~191--224.

\bibitem{BarlesSouganidis1991}
{\sc G.~Barles and P.~E. Souganidis}, {\em Convergence of approximation schemes
  for fully nonlinear second order equations}, Asymptotic Anal., 4 (1991),
  pp.~271--283.

\bibitem{BellettiniChambolleKholmatov2021}
{\sc G.~Bellettini, A.~Chambolle, and S.~Kholmatov}, {\em Minimizing movements
  for forced anisotropic mean curvature flow of partitions with mobilities},
  Proceedings of the Royal Society of Edinburgh: Section A Mathematics, 151
  (2021), pp.~1135--–1170.

\bibitem{BellettiniKholmatov2018}
{\sc G.~Bellettini and S.~Kholmatov}, {\em Minimizing movements for mean
  curvature flow of droplets with prescribed contact angle}, J. Math. Pures
  Appl. (9), 117 (2018), pp.~1--58.

\bibitem{Cao2003}
{\sc F.~Cao}, {\em Geometric curve evolution and image processing}, Springer,
  2003.

\bibitem{CasellesChambolle2006}
{\sc V.~Caselles and A.~Chambolle}, {\em Anisotropic curvature-driven flow of
  convex sets}, Nonlinear Anal., 65 (2006), pp.~1547--1577.

\bibitem{Chambolle2004}
{\sc A.~Chambolle}, {\em {A}n algorithm for {M}ean {C}urvature {M}otion},
  Interfaces Free Bound., 6 (2004), pp.~195--218.

\bibitem{Chambolle2004_TV}
\leavevmode\vrule height 2pt depth -1.6pt width 23pt, {\em {A}n {A}lgorithm for
  {T}otal {V}ariation {M}inimization and {A}pplications}, J.\ Math.\ Imaging
  Vision, 20 (2004), pp.~89--97.

\bibitem{ChambolleDeGennaroMorini2023}
{\sc A.~Chambolle, D.~D. Gennaro, and M.~Morini}, {\em Minimizing movements for
  anisotropic and inhomogeneous mean curvature flows}, Advances in Calculus of
  Variations,  (2023).

\bibitem{ChambolleMoriniNovagaPonsiglione2019}
{\sc A.~Chambolle, M.~Morini, M.~Novaga, and M.~Ponsiglione}, {\em Existence
  and uniqueness for anisotropic and crystalline mean curvature flows}, J.
  Amer. Math. Soc., 32 (2019), pp.~779--824.

\bibitem{ChambolleMoriniPonsiglione2015}
{\sc A.~Chambolle, M.~Morini, and M.~Ponsiglione}, {\em Nonlocal curvature
  flows}, Arch. Ration. Mech. Anal., 218 (2015), pp.~1263--1329.

\bibitem{ChambolleMoriniPonsiglione2017}
\leavevmode\vrule height 2pt depth -1.6pt width 23pt, {\em Existence and
  uniqueness for a crystalline mean curvature flow}, Comm. Pure Appl. Math., 70
  (2017), pp.~1084--1114.

\bibitem{ChambolleNovaga2007}
{\sc A.~Chambolle and M.~Novaga}, {\em Approximation of the anisotropic mean
  curvature flow}, Math. Models Methods Appl. Sci., 17 (2007), pp.~833--844.

\bibitem{ChenGigaGoto1991}
{\sc Y.~G. Chen, Y.~Giga, and S.~Goto}, {\em Uniqueness and existence of
  viscosity solutions of generalized mean curvature flow equations}, J.
  Differential Geom., 33 (1991), pp.~749--786.

\bibitem{PhilippisLaux2018}
{\sc G.~De~Philippis and T.~Laux}, {\em Implicit time discretization for the
  mean curvature flow of outward minimizing sets}, 2018.
\newblock cvgmt preprint.

\bibitem{EkelandTemam}
{\sc I.~Ekeland and R.~Temam}, {\em Convex analysis and variational problems},
  vol.~Vol. 1 of Studies in Mathematics and its Applications, North-Holland
  Publishing Co., Amsterdam-Oxford; American Elsevier Publishing Co., Inc., New
  York, 1976.
\newblock Translated from the French.

\bibitem{EtoGiga2023}
{\sc T.~Eto and Y.~Giga}, {\em On a minimizing movement scheme for mean
  curvature flow with prescribed contact angle in a curved domain and its
  computation}, Annali di Matematica Pura ed Applicata (1923 -),  (2023).

\bibitem{EtoGigaIshii2012}
{\sc T.~Eto, Y.~Giga, and K.~Ishii}, {\em An area-minimizing scheme for
  anisotropic mean-curvature flow}, Adv.\ Differential Equations, 7 (2012),
  pp.~1031--1084.

\bibitem{EtoGigaIshii2012_2}
\leavevmode\vrule height 2pt depth -1.6pt width 23pt, {\em An area minimizing
  scheme for anisotropic mean curvature flow}, Proc. Japan Acad. Ser. A Math.
  Sci., 88 (2012), pp.~7--10.

\bibitem{ES}
{\sc L.~C. Evans and J.~Spruck}, {\em {M}otion of level sets by mean
  curvature.\ {I}}, J.\ Differential Geom., 33 (1991), pp.~635--681.

\bibitem{G}
{\sc Y.~Giga}, {\em Surface evolution equations}, vol.~99 of Monographs in
  Mathematics, Birkh\"{a}user Verlag, Basel, 2006.
\newblock A level set approach.

\bibitem{GigaPozar2016}
{\sc Y.~Giga and N.~Po\v{z}\'{a}r}, {\em A level set crystalline mean curvature
  flow of surfaces}, Adv. Differential Equations, 21 (2016), pp.~631--698.

\bibitem{GigaPozar2018}
\leavevmode\vrule height 2pt depth -1.6pt width 23pt, {\em Approximation of
  general facets by regular facets with respect to anisotropic total variation
  energies and its application to crystalline mean curvature flow}, Comm. Pure
  Appl. Math., 71 (2018), pp.~1461--1491.

\bibitem{GigaPozar2020}
\leavevmode\vrule height 2pt depth -1.6pt width 23pt, {\em Viscosity solutions
  for the crystalline mean curvature flow with a nonuniform driving force
  term}, Partial Differ. Equ. Appl., 1 (2020), pp.~Paper No. 39, 26.

\bibitem{GigaPozar2022}
\leavevmode\vrule height 2pt depth -1.6pt width 23pt, {\em Motion by
  crystalline-like mean curvature: a survey}, Bull. Math. Sci., 12 (2022),
  pp.~Paper No. 2230004, 68.

\bibitem{GilbargTrudinger1983}
{\sc D.~Gilbarg and N.~S. Trudinger}, {\em {E}lliptic partial differential
  equations of second order.\ {S}econd edition}, Springer-Verlag, Berlin, 1983.

\bibitem{HoffmanIlmanenMartinWhite2021}
{\sc D.~Hoffman, T.~Ilmanen, F.~Mart\'{\i}n, and B.~White}, {\em Notes on
  translating solitons for mean curvature flow}, in Minimal surfaces:
  integrable systems and visualisation, vol.~349 of Springer Proc. Math. Stat.,
  Springer, Cham, [2021] \copyright 2021, pp.~147--168.

\bibitem{IshiiLions1990}
{\sc H.~Ishii and P.-L. Lions}, {\em Viscosity solutions of fully nonlinear
  second-order elliptic partial differential equations}, J. Differential
  Equations, 83 (1990), pp.~26--78.

\bibitem{IS}
{\sc H.~Ishii and M.-H. Sato}, {\em Nonlinear oblique derivative problems for
  singular degenerate parabolic equations on a general domain}, Nonlinear
  Anal., 57 (2004), pp.~1077--1098.

\bibitem{IshiiK2014}
{\sc K.~Ishii}, {\em An approximation scheme for the anisotropic and nonlocal
  mean curvature flow}, NoDEA Nonlinear Differential Equations Appl., 21
  (2014), pp.~219--252.

\bibitem{LuckhausSturzenhecker}
{\sc S.~Luckhaus and T.~Sturzenhecker}, {\em Implicit time discretization for
  the mean curvature flow equation}, Calc. Var. Partial Differential Equations,
  3 (1995), pp.~253--271.

\bibitem{Modica1987}
{\sc L.~Modica}, {\em Gradient theory of phase transitions with boundary
  contact energy}, Ann. Inst. H. Poincar\'{e} Anal. Non Lin\'{e}aire, 4 (1987),
  pp.~487--512.

\bibitem{ProtterWeinberger1984}
{\sc M.~H. Protter and H.~F. Weinberger}, {\em Maximum principles in
  differential equations}, Springer-Verlag, New York, 1984.
\newblock Corrected reprint of the 1967 original.

\bibitem{Zhou2018}
{\sc H.~Zhou}, {\em Nonparametric mean curvature type flows of graphs with
  contact angle conditions}, Int. Math. Res. Not. IMRN,  (2018),
  pp.~6026--6069.

\end{thebibliography}
\end{document}